%% file: main.tex
\documentclass[11pt,a4paper,oneside]{amsart}
\include{top}
\usepackage[style=alphabetic, citestyle=alphabetic, sorting=nyvt, maxbibnames=99]{biblatex}
\AtEveryBibitem{\ifentrytype{article}{\clearfield{issn}}{}}
\AtEveryBibitem{\ifentrytype{article}{\clearfield{url}}{}}
\AtEveryBibitem{\ifentrytype{incollection}{\clearfield{url}}{}}
\AtEveryBibitem{\ifentrytype{book}{\clearfield{url}}{}}
\NewBibliographyString{toappear}
\DefineBibliographyStrings{english}{%
  toappear = {to appear},
}

\renewbibmacro*{in:}{%
  \iffieldundef{pubstate}
    {}
    {\printfield{pubstate}%
     \setunit{\addspace}%
     \clearfield{pubstate}}%
  \printtext{%
    \bibstring{in}\intitlepunct}}
\usepackage{csquotes}

\title[Biautomaticity and hierarchical hyperbolicity]{Commensurating HNN-extensions: hierarchical hyperbolicity and biautomaticity}
\usepackage[foot]{amsaddr}
\author{Sam Hughes}
\address[S.~Hughes]{Mathematical Institute, Andrew Wiles Building, Observatory Quarter, University of Oxford, Oxford OX2 6GG, UK}
\email{sam.hughes@maths.ox.ac.uk}

\author{Motiejus Valiunas}
\address[M.~Valiunas]{Instytut Matematyczny, Universytet Wroc{\l}awski, plac Grunwaldzki 2/4, 50-384 Wroc{\l}aw, Poland}
\email{valiunas@math.uni.wroc.pl}

\subjclass{20F65, 20F10, 57K20 (primary), 57M50, 37E30, 20F67 (secondary)}

\date{\today}

\begin{document}

\begin{abstract}
We construct a $\mathrm{CAT}(0)$ hierarchically hyperbolic group (HHG) acting geometrically on the product of a hyperbolic plane and a locally-finite tree which is not biautomatic.  This gives the first example of an HHG which is not biautomatic, the first example of a non-biautomatic $\mathrm{CAT}(0)$ group of flat-rank $2$, and the first example of an HHG which is injective but not Helly.  Our proofs heavily utilise the space of geodesic currents for a hyperbolic surface.
\end{abstract}
\maketitle

\section{Introduction}
Let $H$ be a locally compact group with Haar measure $\mu$.  A discrete subgroup $\Gamma<H$ is a \emph{lattice} if $\mu(H/\Gamma)$ is finite.  We say $\Gamma$ is \emph{uniform} is $H/\Gamma$ is compact.  Roughly speaking, we say a lattice $\Gamma<H_1\times H_2$ is \emph{irreducible} if the projection of $\Gamma\to H_i$ is non-discrete and if $\Gamma$ does not split as a direct product of two infinite groups (see \Cref{sec:prelim} for details).  Throughout we will denote the $n$-regular tree by $\calt_n$ and its automorphism group by $T_n$.

Automatic and biautomatic groups were developed in the 1980s; with a detailed account given in the book \cite{EpsteinEtAl} by Epstein, Cannon, Holt, Levy,  Paterson, and Thurston.  In the 1990s Alonso and Bridson introduced the class of semihyperbolic groups \cite{AlonsoBridson1995} which contains all $\CAT(0)$ and biautomatic groups.  In recent work of Leary and Minasyan \cite{LearyMinasyan2019}, the authors construct irreducible uniform lattices in $\Isom(\EE^{2n})\times T_{2m}$ ($m\geq2$, $n\geq1$), giving the first examples of $\CAT(0)$ groups which are not biautomatic.  These groups were classified up to isomorphism by the second author \cite{Valiunas2020} and studied in the context of fibring by the first author \cite{Hughes2022}.  It follows from \cite{Hughes2021a} and \cite{Valiunas2021a} that all \emph{known} examples of $\CAT(0)$ but not biautomatic groups are either constructed from or contain non-biautomatic Leary--Minasyan groups as subgroups.

In the 2010s, the coarse geometric class of hierarchically hyperbolic groups (HHGs) and spaces (HHSs) were introduced by Behrstock, Hagen and Sisto in \cite{BehrstockHagenSisto2017a} with the motivation coming from isolating the main geometric features common to mapping class groups and compact special groups.  Very roughly these are spaces admitting a coordinate system and hierarchy consisting of and parameterised by hyperbolic spaces, and groups of isometries acting geometrically whilst preserving the hierarchy and coordinate structure.  The theory has received a lot of attention; being studied and developed by numerous authors \cite{BehrstockHagenSisto2017,DurhamHagenSisto2017,DurhamHagenSisto2017,Spriano2018I,Spriano2018II,AbbottBehrstock2019,AbbottNgSpriano2019,FedericoRobbio2019,DurhamMinskySisto2020,RobbioSpriano2020,PetytSpriano2020}.

As previously mentioned some of the main motivation for, and examples of, HHGs come from $\CAT(0)$ cubical groups \cite{BehrstockHagenSisto2017,HagenSusse2020} which are known to be biautomatic by the work of Niblo and Reeves \cite{NibloReeves1998}.  A 2021 result of Haettel, Hoda and Petyt shows that HHGs are semihyperbolic \cite{HaettelHodaPetyt2021}, as a corollary this gave a new proof that mapping class groups are semihyperbolic (see also \cite{DurhamMinskySisto2020} and \cite{Hamenstaedt2009}).  One may hope that proving HHGs are biautomatic would give another proof that mapping class groups are biautomatic.  Thus, a natural question is whether every HHG is biautomatic?  It appears to be open whether any non-biautomatic Leary--Minasyan groups are HHGs---although experts expect them not to be.  In this paper we construct the first example of an HHG which is not biautomatic.

\begin{thmx}\label{thmx.main}
There exists a non-residually finite torsion-free uniform irreducible lattice $\Gamma<\PSL_2(\RR)\times T_{24}$ such that $\Gamma$ is a hierarchically hyperbolic group but is not biautomatic.
\end{thmx}

The group we construct is a ``hyperbolic'' analogue of the groups introduced by Leary--Minasyan in \cite{LearyMinasyan2019}.  Indeed, $\Gamma$ is an HNN-extension of an arithmetic surface where the stable letter commensurates the surface whilst acting as an infinite order elliptic isometry of the hyperbolic plane $\RH^2$.  That the action is by isometries allows us to deduce that $\Gamma$ is a $\CAT(0)$ lattice acting freely cocompactly on the product $\RH^2\times\calt_{24}$, where $\calt_{24}$ is the Bass--Serre tree.  Note that we adopt the lattice viewpoint so we may use results of \cite{Hughes2021a}.  From here we apply \cite[Corollary~3.3]{Hughes2021b} to deduce $\Gamma$ is an HHG.  

Our strategy to show that $\Gamma$ is not biautomatic is very different to Leary--Minasyan's work (for example $\Gamma$ is neither constructed from nor contains a Leary--Minasyan group).  Instead of studying the boundary of a biautomatic structure, we develop a new method to show the failure of biautomaticity. In particular, we use deep work of Mart\'inez-Granado and Thurston on extending functions to the space of geodesic currents of a hyperbolic surface \cite{MartinezGranado2020,MartinezGranadoThurston2020}.

The question of whether every automatic group is biautomatic first appeared in \cite[Question~2.5.6]{EpsteinEtAl} and \cite[Remark~6.19]{GerstenShort1991}.  We do not know if the group $\Gamma$ is automatic.  In spite of this we can still deduce an amusing consequence.

\begin{corx}
At least one of the following statements is false:
\begin{enumerate}
    \item Every HHG is automatic.\label{corx.lol.1}
    \item Every automatic group is biautomatic. \label{corx.lol.2}
\end{enumerate}
\end{corx}

Note that the analogous statement with ``$\CAT(0)$ group'' instead of ``HHG'' follows from the work of Leary and Minasyan \cite{LearyMinasyan2019}.  However, since $\Gamma$ is $\CAT(0)$ it can be deduced here too.

\begin{question}
Is the group $\Gamma$ automatic?
\end{question}

Recall that the \emph{flat-rank} of a $\CAT(0)$ group $\Gamma$ (acting on $X$), denoted $\mathrm{flat}$-$\mathrm{rank}(\Gamma)$, is the maximal rank of an isometrically embedded Euclidean space in $X$.

In \cite[Question~43]{FarbHruskaThomas2011} it was asked if every group acting geometrically on a piece-wise Euclidean $\CAT(0)$ $2$-complex group is biautomatic. (See \cite{FarbHruskaThomas2011} and \cite{MunroOsajdaPrzyycki2021} for recent progress.)
One may hope to relax the hypothesis ``$2$-dimensional piece-wise Euclidean $\CAT(0)$'' to ``flat-rank $2$ $\CAT(0)$''.  Indeed, all previous examples of $\CAT(0)$ but not biautomatic groups have had flat-rank at least $3$.  The next corollary, which follows from the Flat Torus Theorem, shows that one cannot.  

\begin{corx}
There exists a $\CAT(0)$ group $\Gamma$ with $\mathrm{flat}$-$\mathrm{rank}(\Gamma)=2$, that is not biautomatic.
\end{corx}

In \cite{HuangPrytula2020} the authors introduce a property regarding commensurators of abelian subgroups, \emph{Condition (C)}, and show that its failure is closely related to Leary--Minasyan groups \cite[Proposition~8.4]{HuangPrytula2020}.  A natural question would be to ask whether the failure of Condition (C) for a $\CAT(0)$ group is equivalent to the failure of biautomaticity.  However, by \cite[Theorem~1.3]{HuangPrytula2020}, the group $\Gamma$ has Condition (C) and fails to be biautomatic.

In \cite[Theorem~7.3]{BehrstockHagenSisto2019} it is shown that HHGs are coarse median spaces as introduced by Bowditch \cite{Bowditch2013}.  We say a group is a \emph{coarse median group} if it acts geometrically on a coarse median space and the coarse median operator is equivariant up to bounded error.  In \cite[Remark~3.14]{PetytThesis} it is shown that HHGs are coarse median groups.  We remark that $\Gamma$ appears to be the first example of a coarse median group of type $\mathsf{F}$ which is not biautomatic.

The group $\Gamma$ also appears as an example highlighting the difference between discrete and non-discrete versions of ``injective'' metric spaces. We say that a geodesic metric space (respectively a graph) $X$ is \emph{injective} (respectively \emph{Helly}) if the collection of all metric balls in $X$ satisfies the Helly property. Injective metric spaces and Helly graphs, as well as groups acting on them geometrically---\emph{injective groups} and \emph{Helly groups}, respectively---have been extensively studied \cite{Isbell1964,Dress1984,Lang2013,DescombesLang2016,BandeltChepoi2008,ChalopinChepoiGenevoisHiraiOsajda2020,HuangOsajda2021}. The following result gives a negative answer to the question in \cite[Page~4]{Haettel2021}.


\begin{corx}
There exists a group $\Gamma$ which is injective but not Helly and not biautomatic.
\end{corx}

\begin{proof}
It follows from \Cref{thmx.main} that $\Gamma$ is not biautomatic. Moreover, Helly groups are biautomatic \cite[Theorem~1.5(1)]{ChalopinChepoiGenevoisHiraiOsajda2020}, and so $\Gamma$ is not Helly. On the other hand, for every metric space $X$ there exists a ``smallest'' injective metric space $E(X)$, called the \emph{injective hull} of $X$, into which $X$ embeds isometrically, so that a group action on $X$ extends to an action on $E(X)$ \cite{Isbell1964}. It is known that $E(\Htwo)$ is proper and finite Hausdorff distance away from the image of $\Htwo \hookrightarrow E(\Htwo)$ \cite[Proposition~4.6]{Haettel2021}; it also follows from the definitions that (real) trees are injective and that the $\ell_\infty$ product $X \times_\infty Y$ of injective spaces $X$ and $Y$ is injective. Therefore, the geometric action of $\Gamma$ on $\Htwo \times_\infty \calt_{24}$ extends to a geometric action on the proper injective metric space $E(\Htwo) \times_\infty \calt_{24}$, and so $\Gamma$ is injective, as required.
\end{proof}

On the other hand, one may replace the Helly property with a coarse Helly property to study the classes of \emph{coarsely injective} and \emph{coarsely Helly} graphs and groups. Coarsely injective and coarsely Helly groups have been studied in \cite{HaettelHodaPetyt2021,ChalopinChepoiGenevoisHiraiOsajda2020,OsajdaValiunas2020}; in particular, it has been shown that all HHGs are coarsely injective \cite[Corollary~H]{HaettelHodaPetyt2021}. It is currently unknown if all coarsely Helly groups are biautomatic, or even if they all are Helly.
\begin{question}
Is $\Gamma$ coarsely Helly?
\end{question}
It has been communicated to us by Alexander Engel and Damian Osajda that they have shown certain mapping class groups are not Helly. Such groups are HHGs and therefore coarsely injective.


It is a well known open problem whether $S$-arithmetic lattices are biautomatic.  Indeed, this is a special case of \cite[Problem~34]{McCammond2007} in McCammonds list (after the American Institute of Mathematics meeting ‘Problems in Geometric Group Theory’ April 23–27, 2007).  It would be extremely interesting to adapt the methods here to apply to a uniform $S$-arithmetic lattice in $\PSL_2(\RR)\times\PSL_2(\QQ_p)$.  The main issue is showing that vertex stabilisers in the action on the Bruhat--Tits tree $\calt_{p+1}$ are quasi-convex with respect to any biautomatic structure on the lattice.  Note that since such a lattice is residually finite, so if this strategy can be implemented successfully, one would also get a negative answer to \cite[Question~12.4]{LearyMinasyan2019}.

We end with a broad conjecture which would vastly generalise our work here.  The reader is directed to \Cref{sec:prelim} for definitions.

\begin{conjecture}
Let $H$ be a semi-simple real Lie group with trivial centre and no compact factors.  Let $T$ be the automorphism group of a locally-finite unimodular leafless tree.  Suppose $T$ is non-discrete.  If $\Gamma$ is an irreducible non-residually finite uniform $(H\times T)$-lattice, then $\Gamma$ is not biautomatic.
\end{conjecture}

\subsection{Outline of the paper}
In \Cref{sec:prelim} we revise the necessary background on lattices, biautomatic structures, geodesic currents on a hyperbolic surface, and the intersection form.  The remainder of the article is then dedicated to proving \Cref{thmx.main}.

The strategy of the proof of \Cref{thmx.main} is as follows.  We first assume that $\Gamma$ has a biautomatic structure $(B,\calm)$ and consider a biautomatic structure $(A,\call)$ induced by $(B,\calm)$ on a quasi-convex subgroup $G$; here $G$ is a vertex stabiliser in the action of $\Gamma$ on~$\calt$.  The group $G$ acts freely cocompactly on a copy of $\RH^2$ and so can be identified with a subgroup of $\PSL_2(\RR)$, giving rise to a Riemann surface $\Sigma = G \backslash \PSL_2(\RR)$.  The next step is to show that stable word length function $\tau_\call\colon G\to\RR$ with respect to $(A,\call)$ takes only rational values and extends over the space of geodesic currents of $\Sigma$.  Now, the translation length function of $G$ also extends to the space of geodesic currents of $\Sigma$.  Moreover, using the density of the projection of $\Gamma$ to $\PSL_2(\RR)$ we show that both functions agree.  Now, the translation length function takes values which are not rational multiples of each other.  This is a contradiction, and so $\Gamma$ cannot be biautomatic.

In \Cref{sec:stable-lengths} we study stable word length $\tau_\call$ on a biautomatic structure $(A,\call)$ as a function from $G\to\RR$.  The key results, \Cref{prop:stable-len,lem:translation-lengths-same}, imply that for a hyperbolic group $G$ the function takes rational values.

In \Cref{sec:quasi-smoothing} we show that the function $\tau_\call$, viewed as a function on the homotopy classes of closed curves on $\Sigma$, satisfies a technical property known as ``quasi-smoothing'' (see \Cref{prop:quasi-smoothing}). This allows us to extend $\tau_\call$ continuously to the space of geodesic currents of $\Sigma$.

In \Cref{sec:invariant} we complete our study of functions on geodesic currents. The key result, \Cref{prop.criteria}, is that if $t$ is an elliptic isometry of $\Htwo$ commensurating $G$ such that $\langle G,t \rangle$ is dense in $\PSL_2(\RR)$, and if a continuous function $F$ on the space of geodesic currents of $\Sigma$ is in a sense ``$t$-invariant'', then $F(\gamma)$ is a constant multiple of the length of the geodesic representative of $\gamma$, where $\gamma$ is a closed curve on $\Sigma$. In the remaining sections, we construct a lattice $\Gamma < \PSL_2(\RR) \times T$ that will allow us to apply this result for $F = \tau_\call$.

In \Cref{sec:lattices} we study properties of irreducible uniform lattices in $\PSL_2(\RR)\times T$ for sufficiently general trees.  In particular, for a non-residually finite lattice we prove that projection to $\PSL_2(\RR)$ is dense (\Cref{lem:dense}) and that a vertex stabiliser of the action on the tree $\calt$ is quasi-convex with respect to any biautomatic structure (\Cref{prop:quasiconvex}).

In \Cref{sec:example} we construct $\Gamma$; an explicit example of a non-residually finite irreducible uniform lattice in $\PSL_2(\RR)\times T_{24}$ as an HNN-extension.  The key tool is the arithmetic of quaternion algebras which allow us to ensure the stable letter acts on $\RH^2$ as an infinite order elliptic isometry that commensurates the vertex group.  We show that the translation lengths on $\RH^2$ of some elements of a vertex stabiliser in the tree are not rational multiples of each other (\Cref{lem:trlen-irrational}).

In \Cref{sec:main} we prove \Cref{thmx.main}.  In the appendix (\Cref{app:pres}) we detail a presentation of $\Gamma$.

\subsection*{Acknowledgements}
The authors would like to thank Ian Leary and Ashot Minasyan to whom we both owe a great intellectual debt and whose work has been a continued source of inspiration.  The authors are grateful to Thomas Haettel, Mark Hagen and Ashot Minasyan for their comments on an earlier version of this paper.  The first author would once again like to thank Ian Leary for his tireless, kind, and often humorous encouragement during my PhD.  The first author was supported by the Engineering and Physical Sciences Research Council grant number 2127970 and the European Research Council (ERC) under the European Union’s Horizon 2020 research and innovation programme (Grant agreement No. 850930).  Finally, we would like to thank the anonymous referee for their helpful comments.

\section{Preliminaries}
\label{sec:prelim}

\subsection{Lattices and graphs of groups}
Our example will be constructed as a lattice in $\PSL_2(\RR)\times T_{24}$.  To this end we record some definitions and results we will use in \Cref{sec:lattices} and \Cref{sec:example}.

\begin{defn}
Let $H$ be a locally compact topological group with right invariant Haar measure $\mu$.  A discrete subgroup $\Gamma\leq H$ is a \emph{lattice} if the covolume $\mu(H/\Gamma)$ is finite.  A lattice is \emph{uniform} if $H/\Gamma$ is compact and \emph{non-uniform} otherwise.  Let $S$ be a right $H$-set such that for all $s\in S$, the stabilisers $H_s$ are compact and open, then if $\Gamma\leq H$ is discrete the stabilisers in the action of $\Gamma$ on $S$ are finite.
\end{defn}

Let $X$ be a locally finite, connected, simply connected simplicial complex. The group $H=\Aut(X)$ of simplicial automorphisms of $X$ naturally has the structure of a locally compact topological group, where the topology is given by uniform convergence on compacta.

Note that $T$ the automorphism group of a locally-finite tree $\calt$ admits lattices if and only if the group $T$ is unimodular (that is the left and right Haar measures coincide).  In this case we say $\calt$ is \emph{unimodular}.  We say a tree $\calt$ is \emph{leafless} if it has no vertices of valence one.

Two notions of irreducibility for a lattice will feature in this paper. 

\begin{defn}
Let $\calt$ be a locally-finite unimodular leafless tree not isometric to $\RR$ and let $T=\Aut(\calt)$ be non-discrete and cocompact.  Let $\Gamma$ be a uniform $(\PSL_2(\RR)\times T)$-lattice.  We say that $\Gamma$ is \emph{weakly irreducible} if one (and hence both---see \cite[Proposition~3.4]{Hughes2021a}) of the images of the projections $\pi_{\PSL_2(\RR)}:\Gamma\to\PSL_2(\RR)$ and $\pi_T:\Gamma\to T$ are non-discrete.  We say $\Gamma$ is \emph{algebraically irreducible} if there is no finite index subgroup $\Gamma_1\times \Gamma_2$ of $\Gamma$ with $\Gamma_1$ and $\Gamma_2$ infinite.  By \cite[Theorem~4.2]{CapraceMonod2009b}, the two notions of irreducibility are equivalent for a $(\PSL_2(\RR)\times T)$-lattice $\Gamma$.  So if $\Gamma$ is either (and hence both) weakly or algebraically irreducible we will simply state that $\Gamma$ is \emph{irreducible}.
\end{defn}

To construct and study lattices in product with a tree we will utilise the \emph{graph of lattices} construction from \cite{Hughes2021a}.  Before we do this we will define graphs of groups following Bass \cite{Bass1993}.

\begin{defn}
A \emph{graph of groups} $(A,\cala)$ consists of a connected graph $A$ together with some extra data $\cala=(V\cala,E\cala,\Phi\cala)$.  This data consists of \emph{vertex groups} $A_v\in V\cala$ for each vertex $v$, \emph{edge groups} $A_e = A_{\overline{e}}\in E\cala$ for each (oriented) edge $e$, and monomorphisms $(\alpha_{e}:A_e \rightarrow A_{\iota(e)})\in\Phi\cala$ for every oriented edge in $A$.  We will often refer to the vertex and edge groups as \emph{local groups} and the monomorphisms as \emph{structure maps}.
\end{defn}

\begin{defn}
The \emph{path group} $\pi(\cala)$ has generators the vertex groups $A_v$ and elements $t_e$ for each edge $e\in EA$ along with the relations:
\[
    \left\{\begin{array}{c} 
        \text{The relations in the groups }A_v,\\
 t_{\overline{e}}=t_{e}^{-1},\\
 t_e\alpha_{\overline{e}}(g)t_e^{-1}=\alpha_e(g) \text{ for all } e\in EA \text{ and } g\in A_e=A_{\overline{e}}.
     \end{array}\right\}
\]
\end{defn}

\begin{defn}We will often abuse notation and write $\cala$ for a graph of groups.  The \emph{fundamental group of a graph of groups} can be defined in two ways.  Firstly, considering reduced loops based at a vertex $v$ in the graph of groups, in this case the fundamental group is denoted $\pi_1(\cala,v)$ (see \cite[Definition~1.15]{Bass1993}).  Secondly, with respect to a maximal or spanning tree of the graph.  Let $X$ be a spanning tree for $A$, we define $\pi_1(\cala,X)$ to be the group generated by the vertex groups $A_v$ and elements $t_e$ for each edge $e\in EA$ with the relations:
\[
    \left\{\begin{array}{c} 
    \text{The relations in the groups }A_v,\\
    t_{\overline{e}}=t_e^{-1}\text{ for each (oriented) edge }e,\\
    t_e\alpha_{\overline{e}}(g)t_e^{-1}=\alpha_{e}(g)\text{ for all }g\in A_e,\\
    t_e=1\text{ if }e\text{ is an edge in }X.
    \end{array}\right\}
\]
Note that the definitions are independent of the choice of basepoint $v$ and spanning tree $X$ and both definitions yield isomorphic groups so we can talk about \emph{the fundamental group} of $\cala$, denoted $\pi_1(\cala)$.
\end{defn}

We say a group $G$ is \emph{covirtually} isomorphic to $H$ if there exists a finite normal subgroup $N\trianglelefteq G$ such that $G/N\cong H$.  We are now ready to define a graph of $\PSL_2(\RR)$-lattices.

\begin{defn}\label{def.gol}
A \emph{graph of $\PSL_2(\RR)$-lattices} $(A,\cala,\psi)$ is a graph of groups $(A,\cala)$ that is equipped with a morphism of graphs of groups $\psi\colon\cala\to \PSL_2(\RR)$ such that:
\begin{enumerate}
    \item Each local group $A_\sigma\in\cala$ is covirtually a $\PSL_2(\RR)$-lattice and the image $\psi(A_\sigma)$ is a $\PSL_2(\RR)$-lattice;
    \item The local groups are commensurable in $\Gamma=\pi_1(\cala)$ and their images are commensurable in $\PSL_2(\RR)$;
    \item For each $e\in EA$ the element $t_e$ of the path group $\pi(\cala)$ is mapped under $\psi$ to an element of $\Comm_{\PSL_2(\RR)}(\psi_e(A_e))$.
\end{enumerate}
\end{defn}

The relevance of a graph of $\PSL_2(\RR)$-lattices is the following special case of \cite[Theorem~A]{Hughes2021a}.

\begin{thm}\emph{\cite[Theorem~A]{Hughes2021a}}\label{thm.gol}
Let $(A,\cala,\psi)$ be a finite graph of $\PSL_2(\RR)$-lattices with locally-finite unimodular non-discrete Bass-Serre tree $\calt$, and fundamental group $\Gamma$.  Suppose $T=\Aut(\calt)$ admits a uniform lattice.
If each local group $A_\sigma$ is covirtually a uniform $\PSL_2(\RR)$-lattice, and the kernel $\Ker(\psi|_{A_\sigma})$ acts faithfully on $\calt$, then $\Gamma$ is a uniform $(\PSL_2(\RR)\times T)$-lattice and hence a $\CAT(0)$ group.  Conversely, if $\Lambda$ is a uniform $(\PSL_2(\RR)\times T)$-lattice, then $\Lambda$ splits as a finite graph of uniform $\PSL_2(\RR)$-lattices with Bass-Serre tree $\calt$.
\end{thm}

\subsection{Biautomatic structures}

We are interested in studying when a group $G$ is biautomatic; we briefly introduce the necessary definitions and basic results on biautomaticity below, and refer the interested reader to \cite{EpsteinEtAl} for a more comprehensive account.

We remark that the nowadays standard definition of a biautomatic structure that we give below differs from \cite[Definition~2.5.4]{EpsteinEtAl} (see \cite{Amrhein2021} for an explanation).  However, for finite-to-one structures these definitions are equivalent \cite[Theorem~6]{Amrhein2021}.

Let $G$ be a group with a finite generating set $A$. Formally, we view $A$ as a finite set together with a function $\pi_A^0\colon A \to G$ that extends to a surjective monoid homomorphism $\pi_A\colon A^* \to G$, where $A^*$ is the free monoid on $A$; we say that a word $v \in A^*$ \emph{labels} the element $\pi_A(v) \in G$. For simplicity, we will assume that $A$ is symmetric ($\pi_A(A) = \pi_A(A)^{-1}$) and contains the identity ($\pi_A(1) = 1_G$ for an element $1 \in A$). We denote by $d_A$ the combinatorial metric on the Cayley graph $\cay(G,A)$ of $G$.

We study (combinatorial) paths in $\cay(G,A)$. Given a path $p$ in $\cay(G,A)$ and an integer $t \in \{ 0, \ldots, |p| \}$, where $|p|$ is the length of $p$, we denote by $\widehat{p}(t) \in G$ the $t$-th vertex of $p$, so that $\widehat{p}(0)$ and $\widehat{p}(|p|)$ are the starting and ending vertices of $p$, respectively. We further define $\widehat{p}(t) \in G$ for any $t \in \ZZ_{\geq 0} \cup \{\infty\}$ by setting $\widehat{p}(t) = \widehat{p}(|p|)$ whenever $t > |p|$.

\begin{defn} \label{defn:biauto}
Let $G$ be a group with a finite symmetric generating set $A$ containing the identity, and let $\call \subseteq A^*$. We say $(A,\call)$ is a (\emph{uniformly finite-to-one}) \emph{biautomatic structure} on $G$ if
\begin{enumerate}[label=(\roman*)]
    \item $\call$ is recognised by a finite state automaton over $A$;
    \item there exists $N \geq 1$ such that $1 \leq |\pi_A^{-1}(g) \cap \call| \leq N$ for every $g \in G$; and
    \item \label{it:biauto-ft} $\mathcal{L}$ satisfies the ``two-sided fellow traveller property'': there exists a constant $\zeta \geq 1$ such that if $p$ and $q$ are paths in $\cay(G,A)$ labelled by words in $\call$ and satisfying $d_A(\widehat{p}(0),\widehat{q}(0)) \leq 1$ and $d_A(\widehat{p}(\infty),\widehat{q}(\infty)) \leq 1$, then $d_A(\widehat{p}(t),\widehat{q}(t)) \leq \zeta$ for all $t$.
\end{enumerate}
We say $G$ is \emph{biautomatic} if it has some uniformly finite-to-one biautomatic structure.
\end{defn}

The standard notion of a biautomatic structure appearing in the literature (cf \cite{EpsteinEtAl}) is more general than the notion of a uniformly finite-to-one biautomatic structure as defined here. Nevertheless, it can be shown that every biautomatic group (in the sense of \cite{EpsteinEtAl}, for instance) has a uniformly finite-to-one biautomatic structure \cite[Theorem~2.5.1]{EpsteinEtAl} and so is biautomatic in our sense as well. In this paper, we assume all biautomatic structures to be uniformly finite-to-one.

We record the following result for future reference; for part~\ref{it:biauto-consts-states}, it is enough to take $\nu$ to be larger than the number of states in an automaton over $A$ recognising $\call$.
\begin{thm}[D.~B.~A.~Epstein et al.\ {\cite[Lemma~2.3.9 \& Theorem~3.3.4]{EpsteinEtAl}}] \label{thm:biauto-consts}
Let $(A,\call)$ be a biautomatic structure on a group $G$. Then there exists a constant $\nu \geq 1$ with the following properties:
\begin{enumerate}[label=\textup{(\roman*)}]
    \item \label{it:biauto-consts-states} if $v \in A^*$ is a subword of a word $w \in \call$, then there exist $u_1,u_2 \in A^*$ such that $|u_1|,|u_2| \leq \nu$ and $u_1vu_2 \in \call$, and if $v$ is a prefix (respectively suffix) of $w$, then we can take $u_1 = 1$ (respectively $u_2 = 1$);
    \item if $v,w \in \mathcal{L}$ are such that $\pi_A(v) = \pi_A(wa)$ or $\pi_A(v) = \pi_A(aw)$ for some $a \in A$, then $\big| |v| - |w| \big| \leq \nu$; and
    \item any path in $\cay(G,A)$ labelled by a word in $\call$ is a $(\nu,\nu)$-quasi-geodesic.
\end{enumerate}
\end{thm}

The following notion will be crucial in our arguments.
\begin{defn} \label{defn:qconvex}
Let $(A,\call)$ be a biautomatic structure on a group $G$, and let $H \subseteq G$. We say that $H$ is \emph{$\call$-quasiconvex} if there exists a constant $\xi \geq 1$ such that every path in $\cay(G,A)$ starting and ending at vertices of $H$ and labelled by a word in $\call$ belongs to the $\xi$-neighbourhood of $H$.
\end{defn}

The importance of the notion of $\call$-quasiconvexity can be summarised in the following result. It can be extracted from the proofs of \cite[Theorem~3.1 \& Proposition~4.3]{GerstenShort1991} and from \cite[Lemma~2.3.9]{EpsteinEtAl}.
\begin{thm}[S.~Gersten and H.~Short; D.~B.~A.~Epstein et al.] \label{thm:qconvex}
Let $(B,\calm)$ be a biautomatic structure on a group $G$.
\begin{enumerate}[label=\textup{(\roman*)}]
    \item \label{it:qconvex-centra} For any $g_1,\ldots,g_n \in G$, the centraliser $C_G(\{g_1,\ldots,g_n\})$ is $\calm$-quasiconvex.
    \item \label{it:qconvex-biauto} Let $H \leq G$ be an $\calm$-quasiconvex subgroup. Then there exists a biautomatic structure $(A,\call)$ on $H$ and a constant $\kappa \geq 1$ such that if $v \in \calm$ and $w \in \call$ represent the same element of $G$, then $\big| |v| - |w| \big| \leq \kappa$.
\end{enumerate}
\end{thm}

A biautomatic structure $(A,\call)$ on $H \leq G$ appearing in \Cref{thm:qconvex}\ref{it:qconvex-biauto} will be called a biautomatic structure \emph{associated} to $(B,\calm)$.

Finally, we record the following observation.

\begin{lemma} \label{lem:qconvex-fi}
Let $(A,\call)$ be a biautomatic structure on a group $G$, let $H_1 \leq H_2 \leq G$, and suppose that $[H_2:H_1] < \infty$. Then $H_1$ is $\call$-quasiconvex if and only if $H_2$ is $\call$-quasiconvex.
\end{lemma}

\begin{proof}
Note that since $[H_2:H_1] < \infty$, there exists a constant $\lambda \geq 1$ such that $H_2$ belongs to the $\lambda$-neighbourhood of $H_1$ in $\cay(G,A)$. Moreover, let $\zeta \geq 1$ be the constant appearing in \Cref{defn:biauto}\ref{it:biauto-ft}.

Suppose first that $H_1$ is $\call$-quasiconvex, and let $\xi_1 \geq 1$ be the constant appearing in \Cref{defn:qconvex}. Let $p_2$ be a path in $\cay(G,A)$ labelled by a word in $\call$ with $\widehat{p_2}(0),\widehat{p_2}(\infty) \in H_2$. Since $H_2$ belongs to the $\lambda$-neighbourhood of $H_1$, there exist $g_-,g_+ \in H_1$ such that $d_A(g_-,\widehat{p_2}(0)) \leq \lambda$ and $d_A(g_+,\widehat{p_2}(\infty)) \leq \lambda$; moreover, since $\pi_A|_\call$ is surjective, there exists a path $p_1$ in $\cay(G,A)$ labelled by a word in $\call$, starting at $g_-$ and ending at $g_+$. It then follows that $p_2$ is in the $\lambda\zeta$-neighbourhood of $p_1$, and $p_1$ is in the $\xi_1$-neighbourhood of $H_1$. Therefore, $p_2$ is in the $(\lambda\zeta+\xi_1)$-neighbourhood of $H_1$, and so of $H_2$; it follows that $H_2$ is $\call$-quasiconvex, as required.

Conversely, suppose that $H_2$ is $\call$-quasiconvex, and let $\xi_2 \geq 1$ be the constant appearing in \Cref{defn:qconvex}. Then any path in $\cay(G,A)$ labelled by a word in $\call$ and with endpoints in $H_1$ belongs to the $\xi_2$-neighbourhood of $H_2$, and so to the $(\xi_2+\lambda)$-neighbourhood of $H_1$. It follows that $H_1$ is $\call$-quasiconvex, as required.
\end{proof}

\subsection{Geodesic currents}

We now fix a closed orientable Riemannian surface $\Sigma$ of constant curvature $-1$, and let $G = \pi_1(\Sigma)$. We also fix the universal covering map $\widetilde\Sigma \to \Sigma$ and the $G$-action by isometries on $\widetilde\Sigma$. Let $\ogeod{\widetilde\Sigma}$ be the set of oriented (i.e.\ directed) geodesic lines on $\widetilde\Sigma$. Since each such geodesic line is uniquely determined by its endpoints on $\partial\widetilde\Sigma \cong \Sone$, we can topologise $\ogeod{\widetilde\Sigma}$ by identifying it with the open cylinder $\{ (x,y) \in \Sone \times \Sone \mid x \neq y \}$. Note that the $G$-action on $\widetilde\Sigma$ induces an action of $G$ on $\ogeod{\widetilde\Sigma}$.

By an (\emph{oriented}) \emph{curve} on $\Sigma$ we mean a free homotopy class of essential continuous maps $\Sone \to \Sigma$. We denote by $\ocurves{\Sigma}$ the set of all curves on $\Sigma$, which can also be identified with the set of non-trivial $G$-conjugacy classes. Given a primitive curve $\gamma \in \ocurves{\Sigma}$ (meaning that $\gamma \neq \eta^n$ for any $\eta \in \ocurves{\Sigma}$ and $n \geq 2$), we may associate a Borel measure $\mu_\gamma$ on $\ogeod{\widetilde\Sigma}$ as follows: let $\widehat\gamma\colon \Sone \to \Sigma$ be the unique (up to reparametrisation of $\Sone$) geodesic representative of $\gamma$, let $A(\gamma) \subset \ogeod{\widetilde\Sigma}$ be the set of all lifts of $\widehat\gamma$, and let $\mu_\gamma(E) := |E \cap A(\gamma)|$ for any Borel subset $E \subseteq \ogeod{\widetilde\Sigma}$. We may also define this when $\gamma$ is not primitive, by setting $\mu_{\eta^n} := n\mu_\eta$ for primitive $\eta \in \ocurves{\Sigma}$ and $n \geq 2$. By construction, $A(\gamma)$, and so $\mu_\gamma$, is $G$-invariant; moreover, one can see that $A(\gamma)$ is discrete in $\ogeod{\Sigma}$, implying that $\mu_\gamma$ is a Radon measure. This motivates the following definition.

\begin{defn} \label{defn:current}
An (\emph{oriented}) \emph{geodesic current} on $\Sigma$ is a $G$-invariant Radon measure on $\ogeod{\widetilde\Sigma}$. The set of all geodesic currents on $\Sigma$ form a topological space $\ocurrents{\Sigma}$ under the weak* topology: we have $\mu_n \to \mu$ in $\ocurrents{\Sigma}$ if and only if $\int f \dd\mu_n \to \int f \dd\mu$ for all continuous functions $f\colon \ogeod{\widetilde\Sigma} \to \RR$ with compact support. By slightly abusing the notation, we will identify a curve $\gamma \in \ocurves{\Sigma}$ with the corresponding geodesic current $\gamma := \mu_\gamma \in \ocurrents{\Sigma}$, and will therefore view $\ocurves{\Sigma}$ as a subset of $\ocurrents{\Sigma}$.
\end{defn}

It is known that a current $\mu$ is uniquely determined by the values of $\int f \dd\mu$ for compactly supported continuous functions $f\colon \ogeod{\widetilde\Sigma} \to \RR$, as a consequence of the following theorem.
\begin{thm}[Riesz Representation Theorem; see {\cite[Theorem~1.7.13]{MartinezGranado2020}}] \label{thm:riesz-rep}
Let $X$ be a locally compact Hausdorff space, and let $C_c(X)$ be the set of continuous functions $f\colon X \to \RR$ with compact support. For any linear functional $F\colon C_c(X) \to \RR$ such that $F(f) \geq 0$ whenever $f(x) \geq 0$ for all $x \in X$, there exists a unique Radon measure $\mu$ on $X$ such that $F(f) = \int f \dd\mu$ for all $f \in C_c(X)$. In particular, if $\mu$ and $\mu'$ are Radon measures on $X$ such that $\int f \dd\mu = \int f \dd\mu'$ for all $f \in C_c(X)$, then $\mu = \mu'$.
\end{thm}

A core part of this paper is based on studying certain functions $f\colon \ocurves{\Sigma} \to \RR$. The terminology we use below roughly follows the terminology of \cite{MartinezGranadoThurston2020} and \cite{MartinezGranado2020}; however, since the functions we consider are assumed to satisfy the additive union property in the sense of \cite[Definition~1.1]{MartinezGranadoThurston2020}, we are able to make some simplifications to the statements of results. Given two maps $\widehat\gamma_1, \widehat\gamma_2 \colon \Sone \to \Sigma$, a \emph{crossing} of $\widehat\gamma_1$ and $\widehat\gamma_2$ is a pair $(x_1,x_2) \in \Sone \times \Sone$ such that $\widehat\gamma_1(x_1) = \widehat\gamma_2(x_2)$, and a \emph{self-crossing} of $\widehat\gamma\colon \Sone \to \Sigma$ is a pair $(x_1,x_2) \in \Sone \times \Sone$ such that $x_1 \neq x_2$ and $\widehat\gamma(x_1) = \widehat\gamma(x_2)$. A crossing or a self-crossing is \emph{essential} if, roughly speaking, it is unavoidable in a homotopy class: see \cite[Definition~2.6 \& Lemma~2.8]{MartinezGranadoThurston2020}.

\begin{defn} \label{defn:quasi-smoothing}
Let $f\colon \ocurves{\Sigma} \to \RR$.
\begin{enumerate}[label=(\roman*)]
    \item We say $f$ is \emph{homogeneous} if $f(\gamma^n) = n f(\gamma)$ for all $\gamma \in \ocurves{\Sigma}$ and $n \geq 1$.
    \item We say $f$ satisfies the \emph{join quasi-smoothing property} if there exists a constant $\zeta \geq 0$ such that the following holds. Let $(x_1,x_2)$ be an essential crossing of maps $\widehat\gamma_1, \widehat\gamma_2 \colon \Sone \to \Sigma$ representing curves $\gamma_1,\gamma_2 \in \ocurves{\Sigma}$, respectively, and let $\gamma \in \ocurves{\Sigma}$ be the homotopy class of a curve obtained by cutting $\widehat\gamma_i$ at $x_i$ and regluing the four resulting endpoints in a way that respects the orientation of the $\widehat\gamma_i$. Then $f(\gamma) \leq f(\gamma_1)+f(\gamma_2)+\zeta$.
    \item We say $f$ satisfies the \emph{split quasi-smoothing property} if there exists a constant $\zeta \geq 0$ such that the following holds. Let $(x_1,x_2)$ be an essential self-crossing of a map $\widehat\gamma \colon \Sone \to \Sigma$ representing a curve $\gamma \in \ocurves{\Sigma}$, and let $\gamma_1,\gamma_2 \in \ocurves{\Sigma}$ be the homotopy classes of the two curves obtained by cutting $\widehat\gamma$ at $x_1$ and $x_2$ and regluing the four resulting endpoints in a way that respects the orientation of $\widehat\gamma$. Then $f(\gamma_1)+f(\gamma_2) \leq f(\gamma)+\zeta$.
\end{enumerate}
\end{defn}

Given a function $f\colon \ocurves{\Sigma} \to \RR$ that satisfies the join and split quasi-smoothing properties, the following result allows us to construct such a function that is also homogeneous.

\begin{thm}[D.~Mart\'inez-Granado and D.~P.~Thurston {\cite[Theorem~B]{MartinezGranadoThurston2020}}] \label{thm:MGT-limit}
Let $f\colon \ocurves{\Sigma} \to \RR$ be a function satisfying the join and split quasi-smoothing properties. Then the function $\overline{f}\colon \ocurves{\Sigma} \to \RR$ defined by $\overline{f}(\gamma) = \lim_{n \to \infty} f(\gamma^n)/n$ is well-defined, homogeneous, and satisfies the join and split quasi-smoothing properties.
\end{thm}

The main motivation for these definitions arises from the following result that is crucial in our argument.

\begin{thm}[D.~Mart\'inez-Granado and D.~P.~Thurston {\cite[Theorem~A]{MartinezGranadoThurston2020}}] \label{thm:MGT-main}
Let $f\colon \ocurves{\Sigma} \to \RR$ be a homogeneous function satisfying the join and split quasi-smoothing properties. Then $f$ extends to a unique continuous homogeneous function $f\colon \ocurrents{\Sigma} \to \RR$.
\end{thm}

As a consequence of \Cref{thm:MGT-limit,thm:MGT-main}, if a function $f\colon \ocurves{\Sigma} \to \RR$ satisfies the join and split quasi-smoothing properties, then $\overline{f}\colon \ocurves{\Sigma} \to \RR$ extends to a unique continuous homogeneous function $\overline{f}\colon \ocurrents{\Sigma} \to \RR$.

Another property we will use is ``positive linearity''. We say a function $f\colon \ocurrents{\Sigma} \to \RR$ is \emph{positively linear} if $f(c_1\mu_1+c_2\mu_2) = c_1f(\mu_1) + c_2f(\mu_2)$ for all $c_1,c_2 \geq 0$ and $\mu_1,\mu_2 \in \ocurrents{\Sigma}$.

\begin{lemma} \label{lem:poslin}
Let $f\colon \ocurves{\Sigma} \to \RR$ be a homogeneous function satisfying the join and split quasi-smoothing properties. Then the function $f\colon \ocurrents{\Sigma} \to \RR$ given by \Cref{thm:MGT-main} is positively linear.
\end{lemma}

\begin{proof}
Let $\RR_+\ocurves{\Sigma} \subset \ocurrents{\Sigma}$ be the subspace of currents of the form $\sum_i c_i\gamma_i$ for some $c_i \geq 0$ and $\gamma_i \in \ocurves{\Sigma}$. Since $f\colon \ocurves{\Sigma} \to \RR$ is homogeneous, we can extend it to a function $\widehat{f}\colon \RR_+\ocurves{\Sigma} \to \RR$ by setting $\widehat{f}(\sum_i c_i\gamma_i) := \sum_i c_if(\gamma_i)$. The fact that $f\colon \ocurves{\Sigma} \to \RR$ satisfies the join and split quasi-smoothing properties in our terminology implies that $\widehat{f}\colon \RR_+\ocurves{\Sigma} \to \RR$ satisfies quasi-smoothing in the terminology of \cite{MartinezGranadoThurston2020}. In particular, by the uniqueness in \Cref{thm:MGT-main}, the restriction of $f\colon \ocurrents{\Sigma} \to \RR$ to $\RR_+\ocurves{\Sigma}$ coincides with $\widehat{f}$. By the definition of $\widehat{f}$, it therefore follows that $f(c_1\mu_1+c_2\mu_2) = c_1f(\mu_1) + c_2f(\mu_2)$ for all $c_1,c_2 \geq 0$ and $\mu_1,\mu_2 \in \RR_+\ocurves{\Sigma}$. Since $\RR_+\ocurves{\Sigma}$ is dense in $\ocurrents{\Sigma}$ \cite[Proposition~2]{Bonahon1988} and since $f\colon \ocurrents{\Sigma} \to \RR$ is continuous, it follows that $f\colon \ocurrents{\Sigma} \to \RR$ is positively linear, as required.
\end{proof}

\subsection{Intersection numbers}

Finally, we study the intersection numbers between currents. Let $\odgeod{\widetilde\Sigma} \subset \ogeod{\widetilde\Sigma} \times \ogeod{\widetilde\Sigma}$ be the set of pairs $(\gamma_1,\gamma_2)$ of geodesic lines on $\widetilde\Sigma$ that intersect transversely; one can show that $\odgeod{\widetilde\Sigma}$ is a $4$-manifold. The $G$-action on $\widetilde\Sigma$ induces a free and properly discontinuous action on $\odgeod{\widetilde\Sigma}$, and so we may define the quotient $\odgeod{\Sigma} := \odgeod{\widetilde\Sigma}/G$.

\begin{defn} \label{defn:intersection}
Let $\mu_1,\mu_2 \in \ocurrents{\Sigma}$. The product measure $\mu_1 \times \mu_2$ on $\ogeod{\widetilde\Sigma} \times \ogeod{\widetilde\Sigma}$ is $G$-invariant, so it induces a measure $\mu_1 \boxtimes \mu_2$ on $\odgeod{\Sigma}$. The \emph{intersection number} of $\mu_1$ and $\mu_2$, denoted $\iota_\Sigma(\mu_1,\mu_2)$, is the total mass of the measure $\mu_1 \boxtimes \mu_2$. We write $\iota(\mu_1,\mu_2)$ for $\iota_\Sigma(\mu_1,\mu_2)$ when the surface $\Sigma$ is clear.
\end{defn}

For $\gamma_1,\gamma_2 \in \ocurves{\Sigma}$, one may check that $\iota(\gamma_1,\gamma_2)$ is equal to the standard geometric intersection number of geodesic representatives $\widehat{\gamma_1}, \widehat{\gamma_2} \colon \Sone \to \Sigma$, i.e.\ the number of points at which $\widehat{\gamma_1}$ and $\widehat{\gamma_2}$ intersect transversely. Moreover, it turns out that the intersection number is always finite, and induces a continuous function $\ocurrents{\Sigma} \times \ocurrents{\Sigma} \to \RR$:

\begin{thm}[F.~Bonahon {\cite[\S 4.2]{Bonahon1986}}] \label{thm:bonahon-cts}
For any $\mu_1,\mu_2 \in \ocurrents{\Sigma}$, we have $\iota(\mu_1,\mu_2) < \infty$. Moreover, the function $\iota\colon \ocurrents{\Sigma} \times \ocurrents{\Sigma} \to \RR$ is homogeneous and continuous.
\end{thm}

We say $\lambda \in \ocurrents{\Sigma}$ is a \emph{filling current} if every geodesic line in $\widetilde\Sigma$ transversely intersects another geodesic line contained in the support of $\lambda$.

\begin{prop}[F.~Bonahon {\cite[Proposition~4 and its proof]{Bonahon1988}}] \label{prop:bonahon-filling}
Let $\lambda \in \ocurrents{\Sigma}$ be a filling current. Then $\iota(\lambda,\mu) > 0$ for all $\mu \in \ocurrents{\Sigma}$, and the subspace $\{ \mu \in \ocurrents{\Sigma} \mid \iota(\lambda,\mu) \leq 1 \}$ of $\ocurrents{\Sigma}$ is compact.
\end{prop}

We record the following observation on intersection numbers for future reference.

\begin{lemma} \label{lem:iota-covers}
Let $\varphi\colon \Sigma' \to \Sigma$ be a $k$-sheeted covering map (for some $k < \infty$) that is a local isometry, and let $\widetilde\varphi\colon \widetilde{\Sigma'} \to \widetilde\Sigma$ be a lift of $\varphi$. Then $\iota_{\Sigma'}(\mu_1 \circ \widetilde\varphi,\mu_2 \circ \widetilde\varphi) = k \cdot \iota_\Sigma(\mu_1,\mu_2)$ for all $\mu_1,\mu_2 \in \ocurrents{\Sigma}$.
\end{lemma}

\begin{proof}
The isometry $\widetilde\varphi$ induces a homeomorphism $\ogeod{\widetilde{\Sigma'}} \times \ogeod{\widetilde{\Sigma'}} \to \ogeod{\widetilde\Sigma} \times \ogeod{\widetilde\Sigma}$ that maps $\odgeod{\widetilde{\Sigma'}}$ onto $\odgeod{\widetilde\Sigma}$; let $\widetilde\varphi'\colon \odgeod{\widetilde{\Sigma'}} \to \odgeod{\widetilde\Sigma}$ be this induced map. Moreover, we can canonically identify $\odgeod{\Sigma}$ with the set of triples $(x,t_1,t_2)$, where $x \in \Sigma$ and $t_1,t_2 \in T_x^1\Sigma \cong \Sone$ are such that $t_1 \neq t_2$, and the topology is the ``usual'' one (see \cite{Bonahon1986}); this viewpoint allows us to see that $\varphi$ induces a $k$-sheeted covering map $\varphi'\colon \odgeod{\Sigma'} \to \odgeod{\Sigma}$. One may check that we have $p_\Sigma \circ \widetilde\varphi' = \varphi' \circ p_{\Sigma'}$, where $p_\Sigma \colon \odgeod{\widetilde\Sigma} \to \odgeod{\Sigma}$ and $p_{\Sigma'} \colon \odgeod{\widetilde{\Sigma'}} \to \odgeod{\Sigma'}$ are the canonical covering maps. It follows that $(\mu_1 \circ \widetilde\varphi) \boxtimes (\mu_2 \circ \widetilde\varphi) = (\mu_1 \boxtimes \mu_2) \circ \varphi'$, i.e.\ we have $\left[(\mu_1 \circ \widetilde\varphi) \boxtimes (\mu_2 \circ \widetilde\varphi)\right](A) = (\mu_1 \boxtimes \mu_2)(\varphi'(A))$ for every Borel subset $A \subseteq \odgeod{\Sigma'}$ such that $\varphi'|_A$ is injective. This implies that $\left[(\mu_1 \circ \widetilde\varphi) \boxtimes (\mu_2 \circ \widetilde\varphi)\right](\odgeod{\Sigma'}) = k \cdot (\mu_1 \boxtimes \mu_2)(\odgeod{\Sigma})$, as required.
\end{proof}

\section{Stable word lengths} \label{sec:stable-lengths}

Throughout this section, we fix a biautomatic group $G$ with a (uniformly finite-to-one) biautomatic structure $(A,\call)$. We define several functions $G \to \RR$ associated to lengths of words in $\call$, and study the relationship between them.

\begin{prop} \label{prop:stable-len}
Let $g \in G$, and for each $n \geq 1$ let $w_n \in \mathcal{L}$ be a word representing $g^n$. Then the sequence $\left( |w_n| / n \right)_{n=1}^\infty$ converges and the limit $\lim_{n \to \infty} |w_n|/n$ is rational.
\end{prop}

\begin{proof}
Suppose first that $g$ has finite order, and so the subgroup $\langle g \rangle$ is finite. Since $(A,\mathcal{L})$ is finite-to-one, it follows that the set $\{ w_n \mid n \geq 1 \}$ is finite, and so the sequence $\left( |w_n| \right)_{n=1}^\infty$ is bounded. Therefore, $|w_n| / n \to 0 \in \QQ$ as $n \to \infty$, which implies the result.

Suppose now that $g$ has infinite order. By \Cref{thm:qconvex}\ref{it:qconvex-centra}, the centraliser $C_G(g)$ is $\call$-quasiconvex (with associated structure $(A_1,\call_1)$, say), and so finitely generated, implying (again by \Cref{thm:qconvex}) that its centre $Z(C_G(g))$ is $\call_1$-quasiconvex (with associated structure $(A_2,\call_2)$, say). Thus $Z(C_G(g))$ is a finitely generated abelian group containing $g$, and so (as $g$ has infinite order) we have $Z(C_G(g)) = H \times F$, where $H \cong \ZZ^N$, $g \in H$, and $F$ is finite. In particular, $H$ has finite index in $Z(C_G(g))$, and is therefore $\call_2$-quasiconvex; let $(B,\calm)$ be the associated biautomatic structure on $H$.

By applying \Cref{thm:qconvex}\ref{it:qconvex-biauto} three times, it follows that there exists a constant $\kappa \geq 1$ such that for each $n \geq 1$, if $v_n \in \calm$ is a word representing $g^n$ then $\big| |v_n|-|w_n| \big| \leq \kappa$. Moreover, since $(B,\calm)$ is a biautomatic structure on $H \cong \ZZ^N$, there exists a constant $\rho \geq 0$ and a function $f\colon H \to \QQ$ such that $f(h^n) = nf(h)$ for all $h \in H$ and $n \geq 1$, and such that $\big| f(h) - |u| \big| \leq \rho$ for any $h \in H$ and any word $u \in \calm$ representing $h$ \cite[Proposition~4.2 and its proof]{Valiunas2021a}. In particular, for each $n \geq 1$ we have
\[
\left| \frac{|w_n|}{n} - f(g) \right| \leq \left| \frac{|v_n|}{n} - \frac{f(g^n)}{n} \right| + \frac{\kappa}{n} = \frac{\big| f(g^n) - |v_n| \big|+\kappa}{n} \leq \frac{\rho+\kappa}{n},
\]
and so $|w_n| / n \to f(g) \in \QQ$ as $n \to \infty$, as required.
\end{proof}

Motivated by this, we introduce the following terminology.
\begin{defn} \label{defn:biauto-notation}
Let $g \in G$.
\begin{enumerate}[label=(\roman*)]
    \item The \emph{$\call$-word length} of $g$, denoted $\modL{g}$, is the length of the shortest word in $\call$ representing $g$.
    \item The \emph{conjugacy $\call$-word length} of $g$ is defined as $\normL{g} := \min_{h \in G} \modL{hgh^{-1}}$.
    \item The \emph{stable $\call$-word length} of $g$ is defined as $\tau_\call(g) := \lim_{n \to \infty} \modL{g^n}/n$.
\end{enumerate}
\end{defn}

It follows from \Cref{prop:stable-len} that $\tau_\call(g)$ is well-defined: indeed, this number is equal to $\lim_{n \to \infty} \frac{|w_n|}{n}$ in the notation of the Proposition. We make the following easy observation.

\begin{lemma} \label{lem:tau-conjinvariant}
We have $\tau_\call(hgh^{-1}) = \tau_\call(g)$ for all $g,h \in G$.
\end{lemma}

\begin{proof}
Let $\nu \geq 1$ be the constant given in \Cref{thm:biauto-consts}, and let $w_n,v_n \in \call$ be the shortest words representing $g^n,hg^nh^{-1}$, respectively. Then $\big| |w_n|-|v_n| \big| \leq 2\nu\max\{|h|_A,1\}$, and so $\tau_{\mathcal{L}}(g) = \lim_{n \to \infty} |w_n|/n = \lim_{n \to \infty} |v_n|/n = \tau_{\mathcal{L}}(hgh^{-1})$, as required.
\end{proof}

In particular, it follows that $\tau_\call(g) = \min_{h \in G} \tau_\call(hgh^{-1}) = \min_{h \in G} \lim_{n \to \infty} \modL{hg^nh^{-1}}/n$ for any $g \in G$. In the remainder of this section, we prove that if $G$ is hyperbolic then the minimum and the limit in this expression can be swapped, and therefore $\tau_\call(g) = \lim_{n \to \infty} \normL{g^n}/n$ for all $g \in G$.

\begin{lemma} \label{lem:quasigeodesic}
Suppose $G$ is hyperbolic. Then there exist constants $\lambda \geq 1$ and $\varepsilon \geq 0$ satisfying the following property. Let $g \in G$ be an element of infinite order, and let $w \in \mathcal{L}$ be a word representing a conjugate of $g$ with $|w| = \normL{g}$. Then any bi-infinite path in $\cay(G,A)$ labelled by $\cdots w w w \cdots$ is a $(\lambda,\varepsilon)$-quasi-geodesic.
\end{lemma}

\begin{proof}
Let $\nu \geq 1$ be the constant given in \Cref{thm:biauto-consts}. Then there exist constants $\ell,\lambda \geq 1$ and $\varepsilon \geq 0$ such that every $\ell$-local $(\nu,2\nu(\nu+1))$-quasi-geodesic in $\cay(G,A)$ is a $(\lambda,\varepsilon)$-quasi-geodesic \cite[Chapitre~3, Th\'eor\`eme~1.4]{CoornaertDelzantPapadopoulos1990}. Moreover, if $w \in A^*$ is a word representing an infinite order element, then any bi-infinite path in $\cay(G,A)$ labelled by $\cdots w w w \cdots$ is a quasi-geodesic \cite[Proposition~8.21]{GhysHarpe1990}. Since $A$ is finite, there are only finitely many words $w \in \call$ of length $< \ell$; therefore, after increasing $\lambda \geq 1$ and $\varepsilon \geq 0$ if necessary, we may assume that for every word $w \in \call$ with $|w| < \ell$ representing an element of infinite order, a path in $\cay(G,A)$ labelled by $\cdots w w w \cdots$ is a $(\lambda,\varepsilon)$-quasi-geodesic.

It is therefore enough to show the following: if $g \in G$ and $w \in \call$ are such that $w$ represents $g$ and $|w| = \normL{g} \geq \ell$, and if $\gamma \subseteq \cay(G,A)$ is a bi-infinite path labelled by $\cdots w w w \cdots$, then any subpath of $\gamma$ of length $\leq \ell$ is a $(\nu,2\nu(\nu+1))$-quasi-geodesic. Thus, let $\eta \subset \gamma$ be a subpath of length $\leq \ell$ from $h \in \gamma$ to $k \in \gamma$. We aim to show that $|\eta| \leq \nu d_A(h,k)+2\nu(\nu+1)$.

Since $|\eta| \leq \ell \leq |w|$, it follows that $\eta$ is labelled by a subword of $ww$; however, since paths labelled by $w$ are $(\nu,\nu)$-quasi-geodesic, we may without loss of generality assume that $\eta$ is \emph{not} labelled by subword of $w$. Thus $\eta$ is labelled by a word $w_1w_2$, where $w_1$ and $w_2$ are a suffix and a prefix of $w$, respectively. Since $|w_1| + |w_2| = |\eta| \leq \ell \leq |w|$, it follows that $w = w_2 v w_1$ for some $v \in A^*$. Since $v$ is a subword of a word in $\call$, there exist words $u_1,u_2 \in A^*$ of length $\leq \nu$ such that $u_1 v u_2 \in \call$. See \Cref{fig:pf-lem-qg}.

Now let $g' \in G$ be the element represented by $w_1w_2v \in A^*$, so that $g$ and $g'$ are conjugate in $G$, and let $w' \in \call$ be a word representing $g'$. By construction, we have $w',u_1vu_2 \in \call$, and there exist paths in $\cay(G,A)$ labelled by those two words that start distance $\leq d_A(h,k)+|u_1| \leq d_A(h,k)+\nu$ apart and end distance $\leq |u_2| \leq \nu$ apart. Therefore,
\begin{align*}
|w| &= \normL{g} \leq |w'| \leq |u_1vu_2| + \nu(d_A(h,k)+2\nu) \leq |v| + 2\nu + \nu(d_A(h,k)+2\nu) \\
&= |w|-|\eta|+\nu d_A(h,k)+2\nu(\nu+1).
\end{align*}
It follows that $|\eta| \leq \nu d_A(h,k)+2\nu(\nu+1)$, as required.
\end{proof}

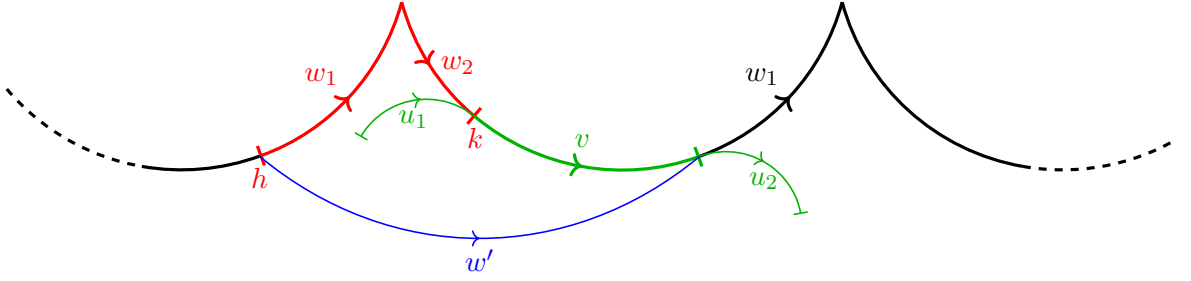
\begin{figure}
    \centering
    \begin{tikzpicture}
    \begin{scope}[very thick,decoration={markings,mark=at position 0.5 with {\arrow{>}}}]
        \draw [dashed] (0,0) arc (-140:-100:3) node (s) {};
        \draw (s.center) arc (-100:-70:3) node (h) {};
        \draw [red,|-,postaction={decorate}] (h.center) node [below] {$h$} arc (-70:-15:3) node [midway,above left] {$w_1$} node (t) {};
        \draw [red,-|,postaction={decorate}] (t.center) arc (-165:-130:3) node [midway,right] {$w_2$} node [below] (k) {$k$};
        \draw [green!70!black,-|,postaction={decorate}] (k.north) arc (-130:-70:3) node [midway,above,yshift=3pt] {$v$} node (h1) {};
        \draw [postaction={decorate}] (h1.center) arc (-70:-15:3) node [midway,above left] {$w_1$} node (t1) {};
        \draw (t1.center) arc (-165:-100:3) node (s1) {};
        \draw [dashed] (s1.center) arc (-100:-60:3);
        \draw [semithick,blue,postaction={decorate}] (h.center) to[bend right=40] node [midway,below] {$w'$} (h1.center);
        \draw [semithick,green!70!black,-|,postaction={decorate}] (h1.center) arc (110:10:1) node [midway,below] {$u_2$};
    \end{scope}
    \draw [semithick,green!70!black,-|,postaction={decorate},decoration={markings,mark=at position 0.5 with {\arrow{<}}}] (k.north) arc (50:150:1) node [midway,below] {$u_1$};
    \end{tikzpicture}
    \caption{The proof of \Cref{lem:quasigeodesic}. The thick path is $\gamma$, the red subpath is $\eta$, and the blue and green paths are labelled by words in $\call$.}
    \label{fig:pf-lem-qg}
\end{figure}

\begin{lemma} \label{lem:translation-lengths-same}
Suppose $G$ is hyperbolic. Then $\normL{g^n}/n \to \tau_\call(g)$ as $n \to \infty$ for every $g \in G$.
\end{lemma}

\begin{proof}
Fix $g \in G$. If $g$ has finite order, then the set $\{ \normL{g^n} \mid n \geq 1 \}$ is bounded, implying that $\normL{g^n}/n \to 0 = \tau_\call(g)$ as $n \to \infty$, as required. Therefore, we may without loss of generality assume that $g$ has infinite order. After replacing $g$ by its conjugate if necessary (we can do this by \Cref{lem:tau-conjinvariant}), we may assume $\normL{g} = |w_1|$ for some word $w_1 \in \mathcal{L}$ representing $g$.

For each $n \geq 1$, let $h_n \in G$ and $w_n \in \call$ be such that $w_n$ represents $h_n^{-1}g^nh_n$ and $\normL{g^n} = |w_n|$. By replacing $h_n$ with $h_ng^M$ for some $M = M(n) \in \ZZ$ if necessary, we may assume that $d_A(h_n,1_G) \leq d_A(h_n,g^m)$ for all $m \in \ZZ$ (when $n$ is fixed); in particular, we may take $h_1 = 1_G$. Let $\gamma_n$ be a bi-infinite path in $\cay(G,A)$ labelled by $\cdots w_n w_n w_n \cdots$ such that a sub-ray of $\gamma_n$ labelled by $w_n w_n w_n \cdots$ starts at the vertex $h_n$.

By \Cref{lem:quasigeodesic}, there exist constants $\lambda \geq 1$ and $\varepsilon \geq 0$ such that each $\gamma_n$ is a $(\lambda,\varepsilon)$-quasi-geodesic. Furthermore, since $\gamma_n$ and $\gamma_1$ contain vertices $g^{mn} h_n$ and $g^{mn}$, respectively, for all $m \in \ZZ$, it follows that $\gamma_n$ and $\gamma_1$ have the same endpoints on the boundary $\partial G$. Therefore, by \cite[Chapitre~3, Th\'eor\`eme~3.1]{CoornaertDelzantPapadopoulos1990}, there exists a constant $\beta \geq 0$ such that $\gamma_n$ is Hausdorff distance $\leq \beta$ away from $\gamma_1$ for each $n \in \ZZ$. See \Cref{fig:pf-lem-tau}.

Now since $h_n \in \gamma_n$, there exists a vertex $k_n \in \gamma_1$ such that $d_A(h_n,k_n) \leq \beta$. Since $\gamma_1$ is the union of the $\langle g \rangle$-translates of a path labelled by $w_1$ and starting at $1_G$, we have $d_A(g^M,k_n) \leq |w_1|$ for some $M = M(n) \in \ZZ$. But then the minimality of $d_A(h_n,1_G)$ implies that
\[
d_A(h_n,1_G) \leq d_A(h_n,g^M) \leq d_A(h_n,k_n) + d_A(g^M,k_n) \leq \beta+|w_1|.
\]
As $A$ is finite, it follows that the set $C := \{ h_n \mid n \geq 1 \}$ is finite.

Now for each $h \in C$ and $n \geq 1$, let $v_{n,h} \in \call$ be a word representing $h^{-1}g^nh \in G$. We then have $\normL{g^n} = \min \{ |v_{n,h}| \mid h \in C \}$. Furthermore, it follows from Proposition~\ref{prop:stable-len} that $|v_{n,h}|/n \to \tau_\call(hgh^{-1})$ and therefore, by Lemma~\ref{lem:tau-conjinvariant}, $|v_{n,h}|/n \to \tau_\call(g)$ as $n \to \infty$, for each $h \in C$. As $C$ is finite, we thus have $\normL{g^n}/n \to \tau_\call(g)$ as $n \to \infty$, as required.
\end{proof}

\begin{figure}
    \centering
    \begin{tikzpicture}
    \begin{scope}[very thick,decoration={markings,mark=at position 0.5 with {\arrow{>}}}]
        \draw [dashed] (0,0) node [left] {$\gamma_1\:$} arc (-120:-60:1) node (s) {};
        \draw (s.center) arc (-60:-30:1) node [above] (g0) {$1_G$};
        \draw [postaction=decorate] (g0.south) arc (-150:-30:1) node [midway,below,yshift=-3pt] {$w_1$} node [above] (g1) {$g$};
        \draw [postaction=decorate] (g1.south) arc (-150:-30:1) node [near start] (k) {} node [midway,below,yshift=-3pt] {$w_1$} node [above] (g2) {$g^2$};
        \draw (g2.south) arc (-150:-90:1) node (s1) {};
        \draw [loosely dotted] (s1.center) -- +(2,0) node (s2) {};
        \draw (s2.center) arc (-90:-30:1) node [above] (gn) {$g^n$};
        \draw [postaction=decorate] (gn.south) arc (-150:-30:1) node [midway,below,yshift=-3pt] {$w_1$} node [above] (gn1) {$g^{n+1}$};
        \draw (gn1.south) arc (-150:-115:1) node (s3) {};
        \draw [dashed] (s3.center) arc (-115:-55:1);
        \draw [dashed] (0,-1.5) node [left] {$\gamma_n\:$} arc (87:81:10.55) node (t) {};
        \draw (t.center) arc (81:70:10.55) node [below] (h) {$h_n$};
        \draw [postaction=decorate] (h.north) arc (110:70:10.55) node [midway,above] {$w_n$} node [below] (gnh) {$g^nh_n$};
        \draw (gnh.north) arc (110:108:10.55) node (t1) {};
        \draw [dashed] (t1.center) arc (108:102:10.55);
        \draw [green!70!black,dashed,semithick] (h.north) to[bend left] node [midway,right] {$\leq\beta$} (k.center);
    \end{scope}
    \end{tikzpicture}
    \caption{The proof of \Cref{lem:translation-lengths-same}.}
    \label{fig:pf-lem-tau}
\end{figure}

\section{Quasi-smoothing}
\label{sec:quasi-smoothing}

Throughout this section, we fix a hyperbolic group $G$ together with a (uniformly finite-to-one) biautomatic structure $(A,\mathcal{L})$ on $G$. We use the notation of \Cref{defn:biauto-notation}.

\begin{lemma} \label{lem:wordlengths}
Let $G$ be a hyperbolic group, and let $(A,\mathcal{L})$ be a finite-to-one biautomatic structure on $G$. Then there exists a constant $\xi \geq 0$ such that the following hold:
\begin{enumerate}[label=\textup{(\roman*)}]
    \item \label{it:len-glue} for all $g,h \in G$, we have $\modL{gh} \leq \modL{g} + \modL{h} + \xi$;
    \item \label{it:len-wiggle} for all $g,h \in G$, we have $\modL{gh} \leq \modL{g} + \xi|h|_A$ and $\modL{hg} \leq \modL{g} + \xi|h|_A$;
    \item \label{it:len-chop} for all $g,h \in G$ and $w \in \mathcal{L}$ such that $w$ represents $gh$ and a prefix of $w$ represents $g$, we have $\modL{g} + \modL{h} \leq \modL{gh} + \xi$.
\end{enumerate}
\end{lemma}

\begin{proof}
Let $\nu$ be the constant given by \Cref{thm:biauto-consts}. Since $\cay(G,A)$ is hyperbolic, there exists a constant $\delta$ such that geodesic triangles in $\cay(G,A)$ are $\delta$-slim, and a constant $\beta$ such that any two $(\nu,\nu)$-quasi-geodesics with the same endpoints are Hausdorff distance $\leq \beta$ away from each other \cite[Chapitre~3, Th\'eor\`eme~1.2]{CoornaertDelzantPapadopoulos1990}. In particular, $(\nu,\nu)$-quasi-geodesic triangles in $\cay(G,A)$ are $(\delta+2\beta)$-slim. We set
\[
\xi := \nu(2\delta+4\beta+4\nu+3).
\]

\begin{enumerate}[label=\textup{(\roman*)}]

    \item Let $v_1,v_2,w \in \mathcal{L}$ be words representing $g$, $h$ and $gh$, respectively, such that $\modL{g} = |v_1|$, $\modL{h} = |v_2|$ and $\modL{gh} = |w|$. Let $\gamma,\zeta_1,\zeta_2 \subseteq \cay(G,A)$ be the paths from $1_G$ to $gh$ (respectively from $1_G$ to $g$, from $g$ to $gh$) labelled by $w$ (respectively $v_1$, $v_2$). Then these three paths form a $(\nu,\nu)$-quasi-geodesic triangle in $\cay(G,A)$, which must be $(\delta+2\beta)$-slim. Thus, if we write $\gamma = \gamma_1\gamma_2$ and $\zeta_1 = \zeta_{11} \zeta_{12}$ in such a way that the endpoints of $\gamma_1$ and $\zeta_{11}$ are distance $\leq \delta+2\beta$ apart and $\gamma_1$ is as long as possible, then we can write $\zeta_2 = \zeta_{21} \zeta_{22}$ in such a way that the starting points of $\gamma_2$ and $\zeta_{22}$ are distance $\leq \delta+2\beta+1$ apart.
    
    Let $w_1,w_2,v_1',v_2' \in A^*$ be the labels of $\gamma_1,\gamma_2,\zeta_{11},\zeta_{22}$, respectively. Then there exist words $u_1,u_2,t_1,t_2 \in A^*$, all of length $\leq \nu$, such that $w_1u_1,u_2w_2,v_1't_1,t_2v_2' \in \mathcal{L}$; see \Cref{subfig:lem-lengths-triangle}. It follows that the endpoints of the paths starting at $1_G$ and labelled by $w_1u_1$ and by $v_1't_1$ are distance $\leq \delta+2\beta+2\nu$ apart, implying that $\big| | w_1u_1 | - | v_1't_1 | \big| \leq \nu(\delta+2\beta+2\nu)$. Similarly, $\big| | u_2w_2 | - | t_2v_2' | \big| \leq \nu(\delta+2\beta+1+2\nu)$. It follows that
    \begin{align*}
        \modL{gh} &= |w| \leq |w_1u_1| + |u_2w_2| \leq |v_1't_1| + |t_2v_2'| + \nu(2\delta+4\beta+4\nu+1) \\
        &\leq |v_1'| + |v_2'| + \nu(2\delta+4\beta+4\nu+1) + 2\nu \leq |v_1| + |v_2| + \xi = \modL{g} + \modL{h} + \xi,
    \end{align*}
    as required.
    
    \item This is trivially true if $h = 1_G$. Otherwise, it is immediate from the choice of $\nu$ that $\modL{gh},\modL{hg} \leq \modL{g} + \nu|h|_A \leq \modL{g} + \xi|h|_A$.
    
    \item Let $w = v_1v_2$, so that $v_1$ and $v_2$ represent $g$ and $h$, respectively. Since $v_1$ and $v_2$ are a prefix and a suffix, respectively, of a word in $\mathcal{L}$, it follows that $v_1u_1,u_2v_2 \in \mathcal{L}$ for some $u_1,u_2 \in A^*$ with $|u_1|,|u_2| \leq \nu$; see \Cref{subfig:lem-lengths-bigon}. It then follows that $\modL{g} \leq |v_1u_1| + \nu|u_1|$ and $\modL{h} \leq |u_2v_2| + \nu|u_2|$. Moreover, we have $\big| \modL{gh} - |w| \big| \leq \nu$, implying that
    \begin{align*}
        \modL{g} + \modL{h} &\leq |v_1u_1| + |u_2v_2| + \nu(|u_1|+|u_2|) = |v_1| + |v_2| + (|u_1|+|u_2|)(\nu+1) \\
        &\leq |w| + 2\nu(\nu+1) \leq \modL{gh} + \nu + \nu(2\nu+2) \leq \modL{gh} + \xi,
    \end{align*}
    as required. \qedhere

\end{enumerate}
\end{proof}

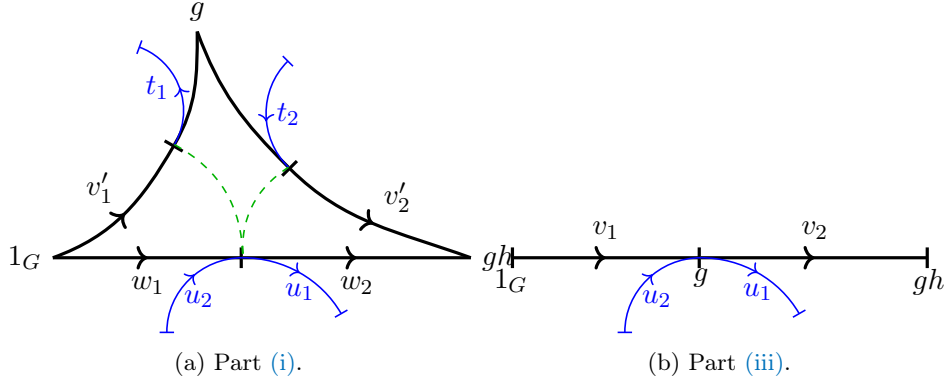
\begin{figure}
    \begin{subfigure}[b]{0.5\textwidth}
    \centering
    \begin{tikzpicture}
    \begin{scope}[very thick,decoration={markings,mark=at position 0.5 with {\arrow{>}}}]
        \draw [-|,postaction=decorate] (0,0) node [left] (id) {$1_G$} to[out=20,in=-120] node [midway,above left] {$v_1'$} (1.6,1.5) node (v1t1) {};
        \draw (v1t1.center) to[out=60,in=-90] (1.9,3) node [above] (g) {$g$};
        \draw (g.south) to[out=-70,in=135] (3.1,1.2) node (t2v2) {};
        \draw [|-,postaction=decorate] (t2v2.center) to[out=-45,in=160] node [midway,above right] {$v_2'$} (5.5,0) node [right] (gh) {$gh$};
        \draw [-|,postaction=decorate] (id.east) to node [midway,below,yshift=-3pt] {$w_1$} (2.5,0) node (w1w2) {};
        \draw [postaction=decorate] (w1w2.center) to node [midway,below,yshift=-3pt] {$w_2$} (gh.west);
        \draw [-|,semithick,blue,postaction=decorate] (v1t1.center) arc (-30:70:0.9) node [midway,left] {$t_1$};
        \draw [-|,semithick,blue,postaction=decorate] (w1w2.center) arc (90:30:1.5) node [midway,below] {$u_1$};
    \end{scope}
    \begin{scope}[very thick,decoration={markings,mark=at position 0.5 with {\arrow{<}}}]
        \draw [-|,semithick,blue,postaction=decorate] (t2v2.center) arc (225:135:1) node [midway,right] {$t_2$};
        \draw [-|,semithick,blue,postaction=decorate] (w1w2.center) arc (90:180:1) node [midway,below] {$\:\:\:u_2$};
    \end{scope}
    \draw [semithick,dashed,green!70!black] (v1t1.center) to[bend left] (w1w2.center) to[bend left] (t2v2.center);
    \end{tikzpicture}
    \caption{Part \ref{it:len-glue}.}
    \label{subfig:lem-lengths-triangle}
    \end{subfigure}%
    \begin{subfigure}[b]{0.5\textwidth}
    \centering
    \begin{tikzpicture}
    \begin{scope}[very thick,decoration={markings,mark=at position 0.5 with {\arrow{>}}}]
        \draw [|-|,postaction=decorate] (0,0) node [below] (id) {$1_G$} to node [midway,above,yshift=3pt] {$v_1$} (2.5,0) node [below] (g) {$g$};
        \draw [-|,postaction=decorate] (g.north) to node [midway,above,yshift=3pt] {$v_2$} (5.5,0) node [below] (gh) {$gh$};
        \draw [-|,semithick,blue,postaction=decorate] (g.north) arc (90:30:1.5) node [midway,below] {$u_1$};
    \end{scope}
    \begin{scope}[very thick,decoration={markings,mark=at position 0.5 with {\arrow{<}}}]
        \draw [-|,semithick,blue,postaction=decorate] (g.north) arc (90:180:1) node [midway,below] {$\:\:\:u_2$};
    \end{scope}
    \end{tikzpicture}
    \caption{Part \ref{it:len-chop}.}
    \label{subfig:lem-lengths-bigon}
    \end{subfigure}
    \caption{The proof of \Cref{lem:wordlengths}. The blue paths have length $\leq \nu$, and the green dashed lines have length $\leq \delta+2\beta+1$.}
    \label{fig:pf-lem-lengths}
\end{figure}

Now let $\Sigma$ be a closed orientable hyperbolic surface, let $G = \pi_1(\Sigma)$, and let $(A,\mathcal{L})$ be a biautomatic structure on $G$ as before. We may then identify $\ocurves{\Sigma}$ with the set of non-trivial conjugacy classes in $G$. Since the function $\normL{-}\colon G \to \RR$ is by definition invariant under conjugacy in $G$, it factors through a function $\ocurves{\Sigma} \to \RR$ which we also denote by $\normL{-}$. We aim to show that $\normL{-}\colon \ocurves{\Sigma} \to \RR$ satisfies the join and split quasi-smoothing properties: see \Cref{defn:quasi-smoothing}.

\begin{prop} \label{prop:quasi-smoothing}
The function $\normL{-}$ satisfies the join and split quasi-smoothing properties.
\end{prop}

\begin{proof}
Let $V = \cay(G,A)/G$: that is, $V$ is a rose---a graph with one vertex---with one loop edge for each element of $A$. We will not distinguish pointed loops in $V$ from their pointed homotopy classes, allowing us to assign to each such loop a label $w \in A^*$. Let $\pi_V\colon \cay(G,A) \to V$ and $\pi_\Sigma\colon \widetilde\Sigma \to \Sigma$ be the canonical covering maps, and let $\theta\colon V \to \Sigma$ be a continuous map that sends each edge in $V$ to a pointed loop on $\Sigma$ labelled by the corresponding element of $A \subset G = \pi_1(\Sigma)$.

Since $\theta \circ \pi_V$ maps loops in $\cay(G,A)$ to nullhomotopic loops in $\Sigma$ and so induces a trivial map $\pi_1(\cay(G,A)) \to \pi_1(\Sigma)$, it follows that $\theta \circ \pi_V = \pi_\Sigma \circ \widetilde\theta$ for some map $\widetilde\theta\colon \cay(G,A) \to \widetilde\Sigma$. Moreover, $\widetilde\theta$ is clearly $G$-equivariant; since $V$ and $\Sigma$ are both compact and the $G$-action on $\widetilde\Sigma$ is properly discontinuous, it follows by the \v{S}varc--Milnor Lemma that $\widetilde\theta$ is a $(\lambda,\varepsilon)$-quasi-isometry for some $\lambda \geq 1$ and $\varepsilon \geq 0$, implying that the diameter of $\widetilde\theta^{-1}(\widetilde{x})$ is at most $\lambda\varepsilon$ for any $\widetilde{x} \in \widetilde\Sigma$. In particular, if $x_1,x_2 \in V$ are such that $\theta(x_1) = \theta(x_2)$, then there are lifts $\widetilde{x}_1 \in \pi_V^{-1}(x_1)$ and $\widetilde{x}_2 \in \pi_V^{-1}(x_2)$ such that $\widetilde\theta(\widetilde{x}_1) = \widetilde\theta(\widetilde{x}_2)$.  This implies that $d_{\cay(G,A)}(\widetilde{x}_1,\widetilde{x}_2) \leq \lambda\varepsilon$; therefore, if $\widetilde\gamma\colon [0,1] \to \cay(G,A)$ is a geodesic from $\widetilde{x}_1$ to $\widetilde{x}_2$, then $\gamma := \pi_V \circ \widetilde\gamma$ is a path in $V$ of length $\leq \lambda\varepsilon$ that is mapped (under $\theta$) to a nullhomotopic loop on $\Sigma$.

Let $\nu$ be the constant given by \Cref{thm:biauto-consts}, and let $\beta \geq 0$ be the constant such that any two $(\nu,\nu)$-quasi-geodesic paths in $\cay(G,A)$ with the same endpoints are Hausdorff distance $\leq \beta$ apart: such a $\beta$ exists by \cite[Chapitre~3, Th\'eor\`eme~1.2]{CoornaertDelzantPapadopoulos1990}. We set
\[
\zeta := \xi\max\{ 9+2\lambda\varepsilon, 7+2\beta+2\lambda\varepsilon \},
\]
where $\xi$ is the constant given in \Cref{lem:wordlengths}. We now prove the \ref{it:join-qs} join quasi-smoothing and \ref{it:split-qs} split quasi-smoothing properties.

\begin{enumerate}[label=(\roman*)]

    \item \label{it:join-qs} For $i \in \{1,2\}$, let $\gamma_i \in \ocurves{\Sigma}$, and let $w_i \in \mathcal{L}$ represent an element in the conjugacy class corresponding to $\gamma_i$ such that $|w_i| = \normL{\gamma_i}$; moreover, let $\sigma_i\colon \Sone \to V$ be the (pointed) loop on $V$ labelled by $w_i$, so that the loop $\widehat{\gamma_i} := \theta \circ \sigma_i$ is in the free homotopy class $\gamma_i$. Suppose $(y_1,y_2)$ is an essential crossing of $\widehat{\gamma_1}$ and $\widehat{\gamma_2}$, so that $\widehat{\gamma_1}(y_1) = \widehat{\gamma_2}(y_2)$, and let $\gamma \in \ocurves{\Sigma}$ be the path obtained by the join quasi-smoothing procedure as in \Cref{defn:quasi-smoothing}. We can thus write $\sigma_1 = \sigma_{11} \cdot \sigma_{12}$ and $\sigma_2 = \sigma_{21} \cdot \sigma_{22}$ for some $\sigma_{ij} \colon [0,1] \to V$, where we write $\sigma' \cdot \sigma''$ for concatenation of paths $\sigma'$ and $\sigma''$ (under some reparametrisation), so that $(\theta \circ \sigma_{11}) \cdot (\theta \circ \sigma_{22}) \cdot (\theta \circ \sigma_{21}) \cdot (\theta \circ \sigma_{12})$ is in the homotopy class $\gamma$.
    
    Let $x_i := \sigma_i(y_i) \in V$ for $i \in \{1,2\}$, so that $\theta(x_1) = \theta(x_2)$. Then, as explained above, there exists a path $\eta\colon [0,1] \to V$ from $x_1$ to $x_2$ of length $\leq \lambda\varepsilon$ such that the loop $\theta \circ \eta$ is nullhomotopic in $\Sigma$. It follows that $\theta \circ (\sigma_{11} \cdot \eta \cdot \sigma_{22} \cdot \sigma_{21} \cdot \overline\eta \cdot \sigma_{12})$ is a well-defined loop that is in the homotopy class $\gamma$, where $\overline\eta\colon [0,1] \to V$ can be taken to be the ``reverse'' of $\eta$. See \Cref{subfig:join-qs}.
    
    By adding or removing initial and terminal subpaths of length at most one to/from the paths $\sigma_{ij}$, $\eta$ and $\overline\eta$, we may modify our construction so that each of these paths start and end at the vertex of $V$. In particular, there exist paths $\sigma_{11}',\sigma_{12}',\sigma_{21}',\sigma_{22}',\eta',\overline\eta'\colon [0,1] \to V$, all starting and ending at the vertex of $V$, such that $\sigma_1 = \sigma_{11}' \cdot \sigma_{12}'$ and  $\sigma_2 = \sigma_{21}' \cdot \sigma_{22}'$, such that $\eta'$ and $\overline\eta'$ have length $\leq \lambda\varepsilon+2$, and such that $\sigma_{11}' \cdot \eta' \cdot \sigma_{22}' \cdot \sigma_{21}' \cdot \overline\eta' \cdot \sigma_{12}'$ is a well-defined loop that is mapped under $\theta$ to the homotopy class $\gamma$.
    
    Let $w_{11},w_{12},w_{21},w_{22},v,\overline{v} \in A^*$ be the labels of the paths $\sigma_{11}',\sigma_{12}',\sigma_{21}',\sigma_{22}',\eta',\overline\eta'$, respectively. We then have $w_1 = w_{11}w_{12} \in \mathcal{L}$, $w_2 = w_{21}w_{22} \in \mathcal{L}$, and $|v|,|\overline{v}| \leq \lambda\varepsilon+2$; moreover, the $G$-conjugacy class of $w_{11}vw_{22}w_{21}\overline{v}w_{12}$ corresponds to the homotopy class $\gamma$. If, given $u \in A^*$, we write $\modL{u}$ for $\modL{g}$, where $g \in G$ is the element represented by $u$, then Lemma~\ref{lem:wordlengths} implies that
    \begin{align*}
        \normL{\gamma} &\leq \modL{w_{11}vw_{22}w_{21}\overline{v}w_{12}} \leq \modL{w_{11}v} + \modL{w_{22}} + \modL{w_{21}\overline{v}} + \modL{w_{12}} + 3\xi \\
        &\leq \modL{w_{11}} + \modL{w_{22}} + \modL{w_{21}} + \modL{w_{12}} + (3+2(\lambda\varepsilon+2))\xi \\
        &\leq \modL{w_{11}w_{12}} + \modL{w_{21}w_{22}} + (3+2(\lambda\varepsilon+2)+2)\xi \leq \normL{\gamma_1} + \normL{\gamma_2} + \zeta,
    \end{align*}
    as required.
    
    \item \label{it:split-qs} Let $\gamma \in \ocurves{\Sigma}$, and let $w \in \mathcal{L}$ represent an element in the conjugacy class corresponding to $\gamma$ such that $|w| = \normL{\gamma}$; moreover, let $\sigma\colon \Sone \to V$ be the (pointed) loop on $V$ labelled by $w$, so that the loop $\widehat{\gamma} := \theta \circ \sigma$ is in the free homotopy class $\gamma$.
    Suppose $(y_1,y_2)$ is an essential self-crossing of $\widehat{\gamma}$, and let $\gamma_1,\gamma_2 \in \ocurves{\Sigma}$ be the paths obtained by the split quasi-smoothing procedure as in \Cref{defn:quasi-smoothing}. Similarly to the previous case (see \Cref{subfig:split-qs}), we may find paths $\sigma_1',\sigma_2',\sigma_3',\eta',\overline\eta' \colon [0,1] \to V$, all starting and ending at the vertex of $V$, such that $\sigma = \sigma_1' \cdot \sigma_2' \cdot \sigma_3'$, such that $\sigma_1' \cdot \eta' \cdot \sigma_3'$ and $\sigma_2' \cdot \overline\eta'$ are well-defined loops that are mapped (under $\theta$) to the free homotopy classes $\gamma_1$ and $\gamma_2$, respectively, and such that $\eta'$ and $\overline\eta'$ have length $\leq \lambda\varepsilon+2$.
    
    Let $w_1,w_2,w_3,v,\overline{v} \in A^*$ be the labels of the paths $\sigma_1',\sigma_2',\sigma_3',\eta',\overline\eta'$, respectively. It then follows that $w = w_1w_2w_3 \in \mathcal{L}$, that $|v|,|\overline{v}| \leq \lambda\varepsilon+2$, and that the $G$-conjugacy classes of $w_1vw_3$ and $w_2\overline{v}$ correspond to the homotopy classes $\gamma_1$ and $\gamma_2$, respectively. Now let $u \in \mathcal{L}$ be a word such that $u$ and $w_1w_2$ represent the same element of $G$. Since $u$ and $w_1w_2$ are both $(\nu,\nu)$-quasi-geodesic words, we can write $u = u_1u_2$ so that $u_1s$ and $w_1$ represent the same element of $G$ for some $s \in A^*$ with $|s| \leq \beta$; consequently, $s^{-1}u_2$ and $w_2$ also represent the same element of $G$. We then have
    \begin{align*}
        \normL{\gamma_1}+\normL{\gamma_2} &\leq \modL{w_1vw_3} + \modL{w_2\overline{v}} \leq \modL{w_1v} + \modL{w_3} + \modL{w_2\overline{v}} + \xi \\
        &\leq \modL{w_1} + \modL{w_3} + \modL{w_2} + (1+2(\lambda\varepsilon+2))\xi \\
        &= \modL{u_1s} + \modL{w_3} + \modL{s^{-1}u_2} + (5+2\lambda\varepsilon)\xi \\
        &\leq \modL{u_1} + \modL{w_3} + \modL{u_2} + (5+2\lambda\varepsilon+2\beta)\xi \\
        &\leq \modL{u} + \modL{w_3} + (5+2\lambda\varepsilon+2\beta+1)\xi \\
        &= \modL{w_1w_2} + \modL{w_3} + (6+2\lambda\varepsilon+2\beta)\xi \\
        &\leq \modL{w_1w_2w_3} + (6+2\lambda\varepsilon+2\beta+1)\xi \leq \normL{\gamma} + \zeta,
    \end{align*}
    as required. \qedhere

\end{enumerate}
\end{proof}

\begin{figure}
    \begin{subfigure}[b]{0.5\textwidth}
    \centering
    \begin{tikzpicture}
    \begin{scope}[very thick,decoration={markings,mark=at position 0.5 with {\arrow{>}}}]
        \draw [|-|,postaction=decorate] (0,0) arc (-180:0:1) node [midway,below,yshift=-3pt] {$\sigma_{11}$} node (s1) {};
        \draw [blue,postaction=decorate] (s1.center) arc (-120:-60:2) node [midway,below,yshift=-3pt] {$\eta$} node (s2) {};
        \draw [|-|,postaction=decorate] (s2.center) arc (-180:0:1) node [midway,below,yshift=-3pt] {$\sigma_{22}$} node (s3) {};
        \draw [postaction=decorate] (s3.center) arc (0:180:1) node [midway,above,yshift=3pt] {$\sigma_{21}$};
        \draw [blue,postaction=decorate] (s2.center) arc (60:120:2) node [midway,above,yshift=3pt] {$\overline\eta$};
        \draw [postaction=decorate] (s1.center) arc (0:180:1) node [midway,above,yshift=3pt] {$\sigma_{12}$};
        \draw [->,semithick] (3,-1.5) to node [midway,left] {$\theta$} (3,-2.5);
        \node at (3,-3.5) (x) {};
        \draw [blue,fill=blue!20,postaction=decorate] (x.center) to[out=130,in=90] (1.5,-3.5) node [left] {$\theta(\eta)$} to[out=-90,in=-130] (x.center);
        \draw [blue,fill=blue!20,postaction=decorate] (x.center) to[out=50,in=90] (4.5,-3.5) node [right] {$\theta(\overline\eta)$} to[out=-90,in=-50] (x.center);
        \draw (x.center) arc (150:205:1.5) node (s0a) {};
        \draw [white] (s0a.center) arc (-155:-145:1.5) node (s0b) {};
        \draw [-|,postaction=decorate] (s0b.center) arc (-145:-30:1.5) node [midway,below,yshift=-3pt] {$\theta(\sigma_{22})$} node (s2) {};
        \draw [postaction=decorate] (s2.center) arc (-30:150:1.5) node [midway,above,yshift=3pt] {$\theta(\sigma_{21})$};
        \draw [-|,postaction=decorate] (x.center) arc (30:210:1.5) node [midway,above,yshift=3pt] {$\theta(\sigma_{12})$} node (s1) {};
        \draw [postaction=decorate] (s1.center) arc (-150:-30:1.5) node [midway,below,yshift=-3pt] {$\theta(\sigma_{11})$} node (s0) {};
        \draw (s0.center) arc (-30:30:1.5);
        \fill [red] (x) circle (2.5pt);
    \end{scope}
    \end{tikzpicture}
    \caption{Join quasi-smoothing, \ref{it:join-qs}.}
    \label{subfig:join-qs}
    \end{subfigure}%
    \begin{subfigure}[b]{0.5\textwidth}
    \centering
    \begin{tikzpicture}
    \begin{scope}[very thick,decoration={markings,mark=at position 0.5 with {\arrow{>}}}]
        \draw [|-|,postaction=decorate] (0.5,0) arc (-180:-60:2.5 and 1) node [pos=0.55,below,yshift=-3pt] {$\sigma_1$} node (s1) {};
        \draw [-|,postaction=decorate] (s1.center) arc (-60:60:2.5 and 1) node [midway,right] {$\sigma_2$} node (s2) {};
        \draw [postaction=decorate] (s2.center) arc (60:180:2.5 and 1) node [pos=0.45,above,yshift=3pt] {$\sigma_3$};
        \draw [blue,postaction=decorate] (s2.center) to [bend left=20] node [midway,right] {$\overline\eta$} (s1.center);
        \draw [blue,postaction=decorate] (s1.center) to [bend left=60] node [midway,left] {$\eta$} (s2.center);
        \draw [->,semithick] (3,-1.5) to node [midway,left] {$\theta$} (3,-2.5);
        \node at (3,-3.5) (x) {};
        \draw [blue,fill=blue!20,postaction=decorate] (x.center) to[out=130,in=90] (1.5,-3.5) node [left] {$\theta(\eta)$} to[out=-90,in=-130] (x.center);
        \draw [blue,fill=blue!20,postaction=decorate] (x.center) to[out=50,in=90] (4.5,-3.5) node [right] {$\theta(\overline\eta)$} to[out=-90,in=-50] (x.center);
        \draw [-|,postaction=decorate] (x.center) arc (30:225:1.5) node [pos=0.8,right] {$\theta(\sigma_3)$} to [out=-45,in=180] (3,-6.42) node (t1) {};
        \draw [postaction=decorate] (x.center) to [out=-120,in=180] (3,-5.2) node [below,yshift=-3pt] {$\theta(\sigma_2)$} to [out=0,in=-60] (x.center);
    \end{scope}
    \begin{scope}[very thick,decoration={markings,mark=at position 0.5 with {\arrow{<}}}]
        \draw [postaction=decorate] (x.center) arc (150:-45:1.5) node [pos=0.8,left] {$\theta(\sigma_1)$} to [out=-135,in=0] (t1.center);
        \fill [red] (x) circle (2.5pt);
    \end{scope}
    \end{tikzpicture}
    \caption{Split quasi-smoothing, \ref{it:split-qs}.}
    \label{subfig:split-qs}
    \end{subfigure}
    \caption{The proof of \Cref{prop:quasi-smoothing}. The top pictures represent the situation in $V$, the bottom ones in $\Sigma$. The red point is the one at which the quasi-smoothing procedure is done, and the blue paths have length $\leq \lambda\varepsilon$.}
    \label{fig:pf-prop-qs}
\end{figure}
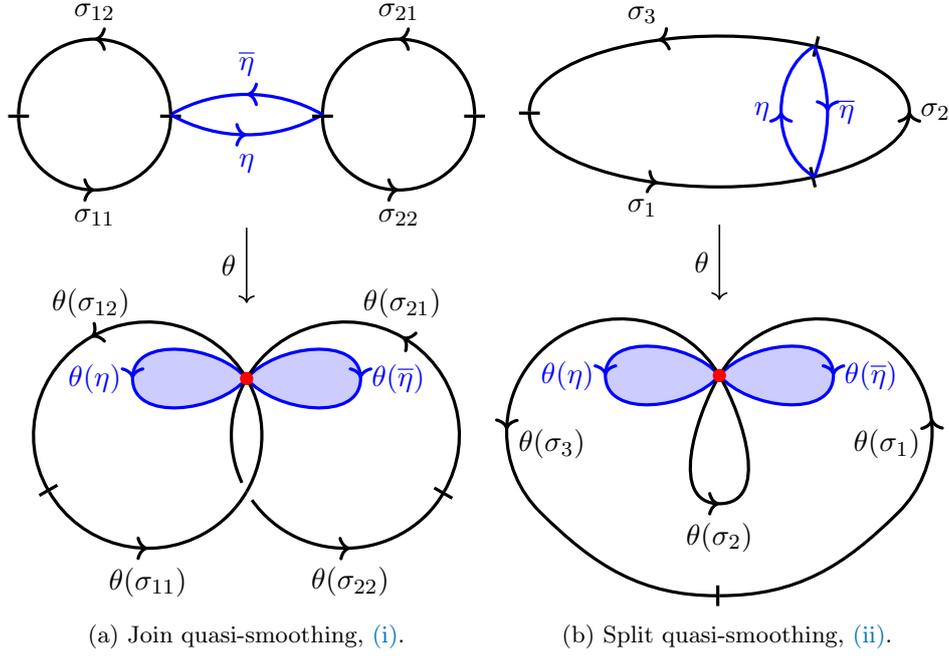

\section{Invariant measures} \label{sec:invariant}

In this section, we fix a closed orientable hyperbolic surface $\Sigma$ and let $G = \pi_1(\Sigma)$. We view $G$ as a uniform lattice in $\PSL_2(\RR) \cong \Isom^+(\widetilde\Sigma)$, the group of orientation-preserving isometries of the universal cover $\widetilde\Sigma \cong \Htwo$ of $\Sigma$.

Note that $\PSL_2(\RR)$ acts smoothly, freely and transitively on $T^1\widetilde\Sigma$, the unit tangent bundle of $\widetilde\Sigma$. We thus have a diffeomorphism $\PSL_2(\RR) \cong T^1\widetilde\Sigma$. Under this diffeomorphism, the $G$-action on $T^1\widetilde\Sigma$ corresponds to the $G$-action on $\PSL_2(\RR)$ by left multiplication, and the $\RR$-action on $T^1\widetilde\Sigma$ by translations along lifts of geodesic lines on $\widetilde\Sigma$ corresponds to the action of a subgroup $R \cong \RR$ on $\PSL_2(\RR)$ by right multiplication; see \cite[\S 1.8.3]{MartinezGranado2020}. After noticing that $\PSL_2(\RR)/R \cong \ogeod{\widetilde\Sigma}$, we can then identify $\ocurrents{\Sigma}$ with a certain space of measures on $\PSL_2(\RR)$, as follows.

\begin{prop}[Y.~Benoist and H.~Oh {\cite[Proposition~8.1]{BenoistOh2007}}] \label{prop:measures-PSL2R}
Let $\mathcal{G}'(\Sigma)$ be the space of Radon measures on $\PSL_2(\RR)$ that are $G$-invariant on the left and $R$-invariant on the right, equipped with the weak* topology. Then the map
\begin{align*}
    \ocurrents{\Sigma} &\to \mathcal{G}'(\Sigma), \\
    \mu &\mapsto \mu',
\end{align*}
where $\mu'(E) = \int \lambda_R(g^{-1}E \cap R) \dd\mu(gR)$ for a Borel subset $E \subseteq \PSL_2(\RR)$ and $\lambda_R$ is a left Haar measure on $R$, is a homeomorphism.
\end{prop}

Throughout this section, we will thus identify $\ocurrents{\Sigma}$ with the space $\mathcal{G}'(\Sigma)$ in \Cref{prop:measures-PSL2R}. We will assume all the measures on $\PSL_2(\RR)$ in this section to be $R$-invariant on the right. We will also fix a left Haar measure $\lambda_\Sigma$ on $\PSL_2(\RR)$. We may rescale $\lambda_\Sigma$ so that $\iota(\gamma,\lambda_\Sigma)$ is equal to the length of the geodesic representative $\widehat{\gamma}\colon \Sone \to \Sigma$ for any $\gamma \in \ocurves{\Sigma}$: see \cite[\S 1.8.3]{MartinezGranado2020}.

Now let $\mu$ be a Radon measure on $\PSL_2(\RR)$ that is $G_0$-invariant for some finite index subgroup $G_0$ of $G$. We then construct a current $\widehat{\mu} \in \ocurrents{\Sigma}$ as follows. Let $g_1,\ldots,g_s$ be a right transversal of $G_0$ in $G$. Given a Borel subset $E \subseteq \PSL_2(\RR)$, we then set
\[
    \widehat{\mu}(E) := s^{-1} \sum_{i=1}^s \mu(g_iE).
\]
It is straightforward to check that $\widehat{\mu}$ is indeed $G$-invariant and does not depend on the choice of the right transversal $G_0$.

We consider the following special case. Let $\mu \in \ocurrents{\Sigma}$, and let $t \in \PSL_2(\RR)$ be such that $G_0 := t^{-1}Gt \cap G$ has finite index in $G$. Then the measure $\mu(t{-})$ is $G_0$-invariant. We define
$\mu^{(t)} := \widehat{\mu'}$, where $\mu' = \mu(t{-})$.

Given an element $t \in \PSL_2(\RR)$, we write $\monoid{G,t}$ for the submonoid of $\PSL_2(\RR)$ generated by $G \cup \{t\}$, and $\monoid{t}$ for the submonoid generated by $t$.

\begin{lemma} \label{lem:monoid-dense}
Let $t \in \PSL_2(\RR)$ be an elliptic isometry of $\Htwo$. If $\langle G,t \rangle$ is dense in $\PSL_2(\RR)$, then so is $\monoid{G,t}$.
\end{lemma}

\begin{proof}
If $t$ has finite order ($m$, say), then we have $t^{-1} = t^{m-1} \in \monoid{t}$ and so $\langle G,t \rangle = \monoid{G,t}$. Therefore, without loss of generality we may assume that $t$ has infinite order.

Since $t$ is elliptic, it stabilises a point $x_0 \in \Htwo$. As $t$ has infinite order, the submonoid $\monoid{t}$ of $\Stab_{\PSL_2(\RR)}(x_0) \cong \Sone$ is infinite, and so dense in $\Stab_{\PSL_2(\RR)}(x_0)$ (by the Dirichlet's Approximation Theorem, for instance). In particular, for any open neighbourhood $U \subseteq \PSL_2(\RR)$ of $t^{-1}$, there exists $m \in \mathbb{N}$ such that $t^m \in U$.

Now let $V \subseteq \PSL_2(\RR)$ be open. Since $\langle G,t \rangle$ is dense, there exists $h \in \langle G,t \rangle$ such that $h \in V$. We can write $h = h_0 t^{-1} h_1 t^{-1} \cdots t^{-1} h_n$ for some $h_0,\ldots,h_n \in \monoid{G,t}$. Consider the map $\varphi\colon \PSL_2(\RR) \to \PSL_2(\RR)$ defined by $\varphi(g) = h_0 g h_1 g \cdots g h_n$, and note that $\varphi(t^{-1}) = h \in V$. Since the multiplication in $\PSL_2(\RR)$ is continuous, so is the map $\varphi$, implying that $\varphi^{-1}(V)$ is an open neighbourhood of $t^{-1}$ in $\PSL_2(\RR)$. But then $t^m \in \varphi^{-1}(V)$ for some $m \in \mathbb{N}$, implying that $\varphi(t^m) \in V \cap \monoid{G,t}$, and in particular that $V \cap \monoid{G,t} \neq \varnothing$. As $V$ was an arbitrary open subset, it follows that $\monoid{G,t}$ is dense in $\PSL_2(\RR)$, as required.
\end{proof}

\begin{lemma} \label{lem:t-invariant=>Liouville}
Let $t \in \PSL_2(\RR)$ be an element such that $G_0 := t^{-1}Gt \cap G$ has finite index in $G$, and such that the monoid $\monoid{G,t}$ is dense in $\PSL_2(\RR)$. Let $\mu \in \ocurrents{\Sigma}$ be a non-zero current such that $\mu^{(t)} = \mu$. Then $\mu = k \cdot \lambda_\Sigma$ for some $k > 0$.
\end{lemma}

\begin{proof}
We aim to show that $\mu(h{-}) = \mu$ for all $h \in \PSL_2(\RR)$: this will imply the result by the uniqueness of the Haar measure.

Let $f\colon \PSL_2(\RR) \to \RR$ be a continuous function with compact support $K \subset \PSL_2(\RR)$, and consider the map $I_f\colon \PSL_2(\RR) \mapsto \RR$ given by $I_f(g) = \int f \dd\mu(g{-})$. Such a map $I_f$ is continuous: see \cite[Lemma~15 on p.~278]{Gaal1973} and its proof.

Now since $\mu$ is $G$-invariant, it follows that $I_f(gh) = I_f(h)$ whenever $g \in G$, implying that $I_f$ factors through the map $\PSL_2(\RR) \to G \backslash \PSL_2(\RR) \cong T^1\Sigma$. Since $T^1\Sigma$ is compact, so is the image of $I_f$, and so $I_f$ attains its infimum: that is, the set
\[
M_f := \{ x \in \PSL_2(\RR) \mid I_f(x) \leq I_f(y) \text{ for all } y \in \PSL_2(\RR) \}
\]
is non-empty.

Now let $g_1,\ldots,g_s$ be a right transversal of $G_0$ in $G$ with $g_1=1$. We then have $\mu(E) = \mu^{(t)}(E) = s^{-1} \sum_{i=1}^s \mu(tg_iE)$ for any Borel subset $E$. In particular, it follows that for any $x \in \PSL_2(\RR)$,
\begin{align*}
I_f(x) &= \int f \dd\mu(x{-}) = s^{-1} \sum_{i=1}^s \int f \dd\mu(tg_ix{-}) = s^{-1} \sum_{i=1}^s I_f(tg_ix).
\end{align*}
Therefore, if $x \in M_f$ then $I_f(tg_ix) = I_f(x)$ for all $i$.  In particular, $I_f(tx) = I_f(x)$: that is, $tx \in M_f$.

Thus, if $x \in M_f$, then $tx \in M_f$ and $gx \in M_f$ for all $g \in G$, implying that $\monoid{G,t} M_f \subseteq M_f$. As $\monoid{G,t}$ is dense in $\PSL_2(\RR)$ and $M_f \neq \varnothing$, it follows that $M_f$ is also dense; as $I_f$ is continuous, this implies that $I_f$ is actually constant on $\PSL_2(\RR)$. But since $f$ was arbitrary, it follows from \Cref{thm:riesz-rep} that $\mu(h{-}) = \mu$ for all $h \in \PSL_2(\RR)$, as required.
\end{proof}

\begin{lemma} \label{lem:somethings-same-intersection}
Let $t \in \PSL_2(\RR)$ be an element such that $G_0 := t^{-1}Gt \cap G$ has finite index in $G$, and let $\mu \in \ocurrents{\Sigma}$. Then $\iota(\mu^{(t)},\lambda_\Sigma) = \iota(\mu,\lambda_\Sigma)$.
\end{lemma}

\begin{proof}
Let $g_1,\ldots,g_s$ be a right transversal of $G_0$ in $G$ with $g_1=1$. For $1 \leq i \leq s$, let $\Sigma_i \to \Sigma$ be the finite covering map corresponding to the subgroup $g_i^{-1}G_0g_i \leq G$, and let $\Sigma_0 \to \Sigma$ be the finite covering map corresponding to the subgroup $tG_0t^{-1} \leq G$. Then the element $tg_i \in \PSL_2(\RR)$ induces an isometry $\varphi_i\colon \Sigma_i \to \Sigma_0$, and also a diffeomorphism $\widetilde\varphi_i\colon \PSL_2(\RR) \to \PSL_2(\RR)$ such that $\mu' \circ \widetilde\varphi_i = \mu'(tg_i-) \in \ocurrents{\Sigma_i}$ for any $\mu' \in \ocurrents{\Sigma_0}$. Note that we have $[G:tG_0t^{-1}] = [G:g_i^{-1}G_0g_i] = s$ for all $i$, since the surfaces $\Sigma_0,\Sigma_1,\ldots,\Sigma_s$ are pairwise isometric (and therefore have the same genus) and since $[G:G_0]=s$.

Now since $\lambda_\Sigma$ is a left Haar measure, we have $\lambda_\Sigma = \lambda_\Sigma(tg_i-) = \lambda_\Sigma \circ \widetilde\varphi_i$ for all $i$. It then follows by \Cref{lem:iota-covers} that
\begin{align*}
    \iota_\Sigma(\mu,\lambda_\Sigma) &= s^{-1} \iota_{\Sigma_0}(\mu,\lambda_\Sigma) = s^{-2}\sum_{i=1}^s \iota_{\Sigma_0}(\mu,\lambda_\Sigma) = s^{-2}\sum_{i=1}^s \iota_{\Sigma_i}(\mu \circ \widetilde\varphi_i,\lambda_\Sigma \circ \widetilde\varphi_i) \\
    &= s^{-2}\sum_{i=1}^s \iota_{\Sigma_i}(\mu(tg_i-),\lambda_\Sigma) = s^{-1}\sum_{i=1}^s (s \cdot [g_i^{-1}G_0g_i:\widehat{G}])^{-1} \iota_{\widehat\Sigma}(\mu(tg_i-),\lambda_\Sigma) \\
    &= [G:\widehat{G}]^{-1} \iota_{\widehat\Sigma}\left( {\textstyle s^{-1}\sum_{i=1}^s \mu(tg_i-)},\lambda_\Sigma \right) = [G:\widehat{G}]^{-1} \iota_{\widehat\Sigma}(\mu^{(t)},\lambda_\Sigma) = \iota_\Sigma(\mu^{(t)},\lambda_\Sigma),
\end{align*}
where $\widehat{G} = \bigcap_{i=1}^s g_i^{-1}G_0g_i$ and $\widehat\Sigma \to \Sigma$ is the finite cover corresponding to $\widehat{G} \leq G$. Thus $\iota(\mu^{(t)},\lambda_\Sigma) = \iota(\mu,\lambda_\Sigma)$, as required.
\end{proof}

\begin{lemma} \label{lem:somethings-t-invariant}
Let $t \in \PSL_2(\RR)$ be an element such that $G_0 := t^{-1}Gt \cap G$ has finite index in $G$. Let $\mu \in \ocurrents{\Sigma}$ be a non-zero current, and define $(\mu_n)_{n=0}^\infty \subset \ocurrents{\Sigma}$ inductively by $\mu_0 = \mu$ and $\mu_n = \mu_{n-1}^{(t)}$ for $n \geq 1$. Then the closure of $\left\{ \sum_{i=0}^n c_i\mu_i \mid n \geq 0, c_i \in [0,\infty) \right\}$ in $\ocurrents{\Sigma}$ contains a non-zero current $\overline\mu$ such that $\overline\mu^{(t)} = \overline\mu$.
\end{lemma}

\begin{proof}

Consider the sequence $(\overline\mu_n)_{n=1}^\infty \subset \ocurrents{\Sigma}$, where $\overline\mu_n = n^{-1} \sum_{i=0}^{n-1} \mu_i$. By \Cref{lem:somethings-same-intersection}, we have $\iota(\mu_{n-1},\lambda_\Sigma) = \iota(\mu_{n-1}^{(t)},\lambda_\Sigma) = \iota(\mu_n,\lambda_\Sigma)$ for all $n \geq 1$, and therefore $\iota(\mu_n,\lambda_\Sigma) = \iota(\mu,\lambda_\Sigma)$ by induction on $n$. It follows that $\iota(\overline\mu_n,\lambda_\Sigma) = \iota(\mu,\lambda_\Sigma)$ for all $n \geq 1$. But by \Cref{prop:bonahon-filling}, the subspace $\{ \mu' \in \ocurrents{\Sigma} \mid \iota(\mu',\lambda_\Sigma) \leq \iota(\mu,\lambda_\Sigma) \}$ of $\ocurrents{\Sigma}$ is compact, implying that the sequence $(\overline\mu_n)_{n=1}^\infty$ has a convergent subsequence: $\overline\mu_{n_m} \to \overline\mu$ as $m \to \infty$, say. Note that we have $\iota(\overline\mu,\lambda_\Sigma) = \lim_{m \to \infty} \iota(\overline\mu_{n_m},\lambda_\Sigma) = \iota(\mu,\lambda_\Sigma)$ since $\iota(-,\lambda_\Sigma)$ is continuous (by \Cref{thm:bonahon-cts}), whereas $\iota(\mu,\lambda_\Sigma) > 0$ since $\mu \neq 0$ and $\lambda_\Sigma$ is filling (by \Cref{prop:bonahon-filling}), so $\overline\mu \neq 0$. We aim to show that $\overline\mu^{(t)} = \overline\mu$.

Let $g_1,\ldots,g_s$ be a right transversal of $G_0$ in $G$ with $g_1=1$, and for $1 \leq i \leq s$, let $\Sigma_i \to \Sigma$ be the finite covering map corresponding to the subgroup $g_i^{-1}G_0g_i \leq G$. Note that since $\overline\mu_{n_m} \to \overline\mu$ as $m \to \infty$, we also have $\overline\mu_{n_m}(tg_i{-}) \to \overline\mu(tg_i{-})$ in $\ocurrents{\Sigma_i}$, and therefore $\overline\mu_{n_m}^{(t)} \to \overline\mu^{(t)}$ as $m \to \infty$. We aim to show that we also have $\overline\mu_{n_m}^{(t)} \to \overline\mu$ as $m \to \infty$; this will imply that $\int f \dd\overline\mu = \int f \dd\overline\mu^{(t)}$ for every continuous function $f\colon \PSL_2(\RR) \to \RR$ with compact support, and the result will then follow by \Cref{thm:riesz-rep}.

Let $D$ be a (relatively compact) fundamental domain for the action of $G$ on $\PSL_2(\RR)$ by left multiplication, and let $K \subset \PSL_2(\RR)$ be compact. We claim that $\mu_n(K) \leq \mu(DK)$ for all $n \geq 0$. Indeed, since we have $\mu_n = \mu_{n-1}^{(t)} = s^{-1} \sum_{i=1}^s \mu_{n-1}(tg_i-)$ for all $n \geq 1$ and since $\mu_0 = \mu$, it follows by induction on $n$ that $\mu_n = s^{-n} \sum_{i=1}^{s^n} \mu(h_i{-})$ for some $h_1,\ldots,h_{s^n} \in \PSL_2(\RR)$. We can pick some $k_1,\ldots,k_{s^n} \in G$ such that $k_ih_i \in D$ for each $i$. Note that $\mu(k_ih_i{-}) = \mu(h_i{-})$ since $\mu$ is $G$-invariant; therefore,
\[
\mu_n(K) = s^{-n} \sum_{i=1}^{s^n} \mu(k_ih_iK) \leq s^{-n} \sum_{i=1}^{s^n} \mu(DK) = \mu(DK),
\]
as claimed.

Now let $f\colon \PSL_2(\RR) \to \RR$ be a continuous function with compact support $K$. Since $\overline\mu_n^{(t)} - \overline\mu_n = n^{-1}(\mu_n-\mu)$ for any $n \geq 1$, we have
\begin{align*}
\left| \int f \dd\overline\mu_n^{(t)} - \int f \dd\overline\mu_n \right| &= n^{-1}\left| \int f \dd\mu_n - \int f \dd\mu \right| \leq n^{-1}\left( \left| \int f \dd\mu_n \right| + \left| \int f \dd\mu \right| \right) \\ &\leq n^{-1} \left\| f \right\|_\infty \left( \mu_n(K) + \mu(K) \right) \leq n^{-1} \left\| f \right\|_\infty \left( \mu(DK) + \mu(K) \right),
\end{align*}
and therefore $\left| \int f \dd\overline\mu_n^{(t)} - \int f \dd\overline\mu_n \right| \to 0$ as $n \to \infty$. On the other hand, since $\overline\mu_{n_m} \to \overline\mu$ we have $\left| \int f \dd\overline\mu_{n_m} - \int f \dd\overline\mu \right| \to 0$ as $m \to \infty$. Since $n_m \to \infty$ as $m \to \infty$, it follows that $\left| \int f \dd\overline\mu_{n_m}^{(t)} - \int f \dd\overline\mu \right| \to 0$ as $m \to \infty$. But as $f$ was arbitrary, it follows that indeed $\overline\mu_{n_m}^{(t)} \to \overline\mu$ (in the weak* topology) as $m \to \infty$, as required.
\end{proof}

\begin{prop}\label{prop.criteria}
Let $t \in \PSL_2(\RR)$ be an elliptic isometry of $\Htwo$ such that $t^{-1} G t \cap G$ has finite index in $G$ and such that $\langle G,t \rangle$ is dense in $\PSL_2(\RR)$. Let $F\colon \ocurrents{\Sigma} \to [0,\infty)$ be a continuous positively linear function such that $F(\gamma^{(t)}) = F(\gamma)$ for all $\gamma \in \ocurves{\Sigma}$. Then $F = k \cdot \iota(-,\lambda_\Sigma)$ for some $k \geq 0$.
\end{prop}

\begin{proof}
Let $k = F(\lambda_\Sigma)/\iota(\lambda_\Sigma,\lambda_\Sigma)$. We will aim to show that $F(\gamma) = k \cdot \iota(\gamma,\lambda_{\Sigma})$ for all $\gamma \in \ocurves{\Sigma}$. As $\RR_+\ocurves{\Sigma}$ is dense in $\ocurrents{\Sigma}$ \cite[Proposition~2]{Bonahon1988} and as $F$ and $\iota(-,\lambda_\Sigma)$ are positively linear and continuous, this will imply the result.

Let $\gamma \in \ocurves{\Sigma}$, and define $(\gamma_n)_{n=0}^\infty$ inductively by $\gamma_0 = \gamma$ and $\gamma_n = \gamma_{n-1}^{(t)}$ for $n \geq 1$. Since $F$ is positively linear and $F(\gamma'^{(t)}) = F(\gamma')$ for any $\gamma' \in \ocurves{\Sigma}$, it follows that $F(\gamma_n) = F(\gamma)$ for all $n \geq 0$; on the other hand, $\iota(\gamma_n,\lambda_\Sigma) = \iota(\gamma,\lambda_\Sigma)$ for all $n \geq 0$ by \Cref{lem:somethings-same-intersection} and induction on $n$. Since $f$ and $\iota(-,\lambda_\Sigma)$ are positively linear, it also follows that $F(\mu)/\iota(\mu,\lambda_\Sigma) = F(\gamma)/\iota(\gamma,\lambda_\Sigma)$ whenever $\mu = \sum_{i=0}^n c_i \gamma_i$ for some $c_0,\ldots,c_n \geq 0$. Since $F$ and $\iota(-,\lambda_\Sigma)$ are continuous, it follows from \Cref{lem:somethings-t-invariant} that there exists a non-zero current $\overline\mu \in \ocurrents{\Sigma}$ such that $F(\overline\mu)/\iota(\overline\mu,\lambda_\Sigma) = F(\gamma)/\iota(\gamma,\lambda_\Sigma)$ and $\overline\mu = \overline\mu^{(t)}$.

Now by \Cref{lem:monoid-dense}, the submonoid $\monoid{G,t}$ is dense in $\PSL_2(\RR)$. Therefore, it follows from \Cref{lem:t-invariant=>Liouville} that $\overline\mu = k' \cdot \lambda_\Sigma$ for some $k' > 0$. Thus
\[
F(\gamma) = \frac{F(\overline\mu)}{\iota(\overline\mu,\lambda_\Sigma)} \iota(\gamma,\lambda_\Sigma) = \frac{k' \cdot F(\lambda_\Sigma)}{k' \cdot \iota(\lambda_\Sigma,\lambda_\Sigma)} \iota(\gamma,\lambda_\Sigma) = k \cdot \iota(\gamma,\lambda_\Sigma),
\]
as required.
\end{proof}

\section{Lattices in the hyperbolic plane and a tree}\label{sec:lattices}
In this section we will collect a number of results about irreducible cocompact lattices in the product of $\PSL_2(\RR)$ and the automorphism group $T$ of a locally-finite unimodular leafless tree $\calt$.  Will assume $T$ is non-discrete.  Throughout $\Gamma$ will be an irreducible cocompact lattice in $\PSL_2(\RR)\times T$.  Note that by \cite[Corollary~3.6]{Hughes2021a} $\Gamma$ is either an irreducible $S$-arithmetic lattice and $\calt$ is a $(p+1)$-regular tree for some prime $p$, or $\Gamma$ is non-residually finite.  In either case, by \Cref{thm.gol} $\Gamma$ contains a commensurated subgroup $G$ isomorphic to the fundamental group of a closed compact surface which arises as a finite index subgroup of a vertex stabiliser of the action of $\Gamma$ on $\calt$.

First, we will investigate the density of the projection of $\Gamma$ to $\PSL_2(\RR)$.

\begin{lemma} \label{lem:dense}
The projection $P$ of $\Gamma$ to $\PSL_2(\RR)$ is dense.
\end{lemma}
\begin{proof}
If $\Gamma$ is linear then $P$ contains an $S$-arithmetic lattice and such a subgroup of $\PSL_2(\RR)$ is dense.  Thus, we may assume $\Gamma$ is non-residually finite. By \Cref{thm.gol} $\Gamma$ splits as a graph of groups in which each vertex group is a finite extension of a uniform lattice in $\PSL_2(\RR)$.  In particular, $P$ contains a uniform $\PSL_2(\RR)$-lattice and hence is Zariski-dense in $\PSL_2(\RR)$.  A Zariski-dense subgroup of $\SL_2(\RR)$ is either dense or discrete.  Indeed the Lie algebra of its closure is an ideal, hence either $0$ or $\mathfrak{sl}_2(\RR)$.  Now, since $\Gamma$ is irreducible, $P$ is non-discrete and so we conclude that $P$ is dense in $\PSL_2(\RR)$.
\end{proof}

Our next task is to show there is a commensurated surface subgroup of $\Gamma$ which is $\calm$-quasiconvex with respect to any biautomatic structure $(B,\calm)$.  The key fact is that in a biautomatic group the centraliser of a finite set is $\calm$-quasiconvex (see \Cref{thm:qconvex}\ref{it:qconvex-centra}).  Before this we will need a lemma.

\begin{lemma}\label{lem:freebylinear}
If $\Gamma$ is non-residually finite, then we have a short exact sequence
\[\begin{tikzcd}\arrow[r] 1 & \arrow[r] F & \arrow[r,"\pi_{\PSL_2(\RR)}"] \Gamma & \arrow[r] P & 1  \end{tikzcd}\]
where $F$ is fundamental group of a graph of finite groups and $P$ is linear.  In particular, if $\Gamma$ is torsion-free, then $F$ is a free group.  In both cases $F$ is infinite, not virtually abelian, and every locally finite subgroup of $F$ is finite.
\end{lemma}
\begin{proof}
Since $\Gamma$ is non-residually finite $\Gamma$ does not admit any faithful linear representation and so $F$ is non-trivial.  Now, $\Gamma$ splits as a graph of finite-by-Fuchsian groups and each Fuchsian group is isomorphic to its image in $P$.  It follows that the action of $F$ on $\calt$ has finite stabilisers.  In particular, $F$ is the fundamental group of a graph of finite groups. If $\Gamma$ is torsion-free, then each vertex and edge stabiliser of the $F$-action on $\calt$ is trivial.  It follows that $F$ admits a free action on a tree and so must be free.  That $P$ is linear follows from the fact $\PSL_2(\RR)$ is linear.  

Since $\Gamma$ is $\CAT(0)$, it has only finitely many conjugacy classes of finite subgroups, implying that any ascending sequence of finite subgroups of $\Gamma$ terminates. It follows that $\Gamma$ (and so $F$) has no infinite locally finite subgroups. We claim that if $F$ was finite then $F$ must act trivially on $\calt$.  Indeed, if $F$ was finite then it acts on $\calt$ elliptically with fixed point set $\calt^F$ a subtree of $\calt$.  By normality of $F$ in $\Gamma$, the subtree $\calt^F$ is $\Gamma$-invariant.  But $\Gamma$ is a uniform lattice and $\calt$ is leafless, so $\Gamma$ acts minimally on $\calt$.  Thus, $F$ is infinite. It remains to show that $F$ is not virtually abelian.

Since $F$ is infinite and not locally finite, it contains a finitely generated infinite subgroup. Such a subgroup cannot be torsion (otherwise it would fix a point in $\calt$, contradicting the fact that the action of $F$ on $\calt$ has finite stabilisers); it follows that $F$ contains an infinite order element $g$. Since the action of $F$ on $\calt$ has finite stabilisers, $g$ must be hyperbolic in this action; let $\ell \subseteq \calt$ be the axis of $g$. Since $T$ is non-discrete it follows that $\calt$ is not a line; moreover, since the $\Gamma$-action on $\calt$ is cocompact and since $\calt$ is leafless and locally-finite, there exists an element $h \in \Gamma$ such that $h\ell \neq \ell$. Then $g$ and $hgh^{-1}$ are two hyperbolic elements of $F$ that have distinct axes, so $g^n$ does not commute with $hg^mh^{-1}$ for any $n,m \neq 0$. This implies that $F$ is not virtually abelian.
\end{proof}

\begin{prop} \label{prop:quasiconvex}
Suppose $(B,\calm)$ is a finite-to-one biautomatic structure on $\Gamma$.  If $\Gamma$ is non-residually finite and torsion-free, then any vertex stabiliser of the action on $\calt$ is an $\calm$-quasiconvex subgroup.
\end{prop}
\begin{proof}
Let $G$ be a commensurated surface subgroup of $\Gamma$.  Let $F=\Ker(\pi_{\PSL_2(\RR)})$ and note by \Cref{lem:freebylinear} that $F$ is a non-abelian free subgroup acting freely on $\calt$. Let $g,h$ be contained in this free group and suppose that they do not commute.  

We claim since $G$ is commensurated and $F$ is normal, the elements $g$ and $h$ commute with the subgroup $S=G\cap G^g\cap G^h$ which has finite index in $G$.  Indeed, let $s\in S$ and note $s$ and $s^g$ fix vertices of $\calt$.  It follows that the commutator $[s,g]$ lies in $G^g$.  The commutator $[s,g]$ maps trivially under the projection to $\PSL_2(\RR)$, but the projection restricted to $G^g$ is injective.  Thus, $[s,g]=1$ and the claim follows. 

By the previous claim, $C:=C_\Gamma(\{g,h\})$ contains $S$.  Now, $g$ and $h$ have distinct axes so $C$ must fix a vertex on $\calt$.  In particular, $C$ is a finite-index subgroup of a vertex stabiliser containing a finite index subgroup $S$ of $G$, implying that $C$ is commensurable with $G$ in $\Gamma$. Finally, since $C$ is the centraliser of a finite set if follows from \Cref{thm:qconvex} that $C$ is $\calm$-quasiconvex. By \Cref{lem:qconvex-fi}, it follows that $G$ is $\calm$-quasiconvex as well.
\end{proof}

Finally, we record this proposition for later use.  It is a special case of \cite[Corollary~3.3]{Hughes2021b}.

\begin{prop}\label{prop.HHG}
$\Gamma$ is a hierarchically hyperbolic group.
\end{prop}

\section{An explicit example}\label{sec:example}
Throughout this section we will use quaternion algebras and arithmetic Fuchsian groups derived from them, for the relevant background the reader should consult \cite[Chapter~5]{Katok1992}.  The construction appeared in the first author's PhD thesis, however, the example there is different to the one given here \cite[Section~4.5.2]{HughesThesis}.

Let $Q$ be the quaternion algebra $(2,13)_\QQ$, this is a $4$-dimensional algebra over $\QQ$ with basis $\{1,i,j,k\}$ satisfying the relations $i^2=2$, $j^2=13$ and $k=ij=-ji$.  The algebra $Q$ has a representation $\varphi\colon Q\to\mathbf{M}_2(\RR)$ given by
\[
   1\mapsto\begin{bmatrix}1&0\\0&1\end{bmatrix}, \quad 
   i\mapsto\begin{bmatrix}\sqrt{2} &0 \\ 0 & -\sqrt{2} \end{bmatrix}, \quad
    j\mapsto\begin{bmatrix} 0&1\\ 13&0\end{bmatrix}, \quad
    k\mapsto\begin{bmatrix}0&\sqrt{2}\\-13\sqrt{2}&0\end{bmatrix}.
\]

Let $t=\frac{1}{3}(1+3i+k)\in Q$, and note that
\[\varphi(t)=\begin{bmatrix}\frac{1}{3}(3\sqrt{2} + 1)    &    \frac{1}{3}\sqrt{2}\\
   -\frac{13}{3}\sqrt{2} & \frac{1}{3}(1-3\sqrt{2})\end{bmatrix}; \]
   it follows that the image of $\varphi(t)$ in $\PSL_2(\RR)$ is an infinite order elliptic isometry of $\RH^2$.
   A basis for a maximal order $M$ of $Q$ is given by the following quaternions
\[\{a,b,c,d\}:=\left\{\frac{3}{2} + \frac{3}{2}i - \frac{1}{2}j - \frac{1}{2}k,\ \frac{3}{2} - \frac{3}{2}i - \frac{1}{2}j + \frac{1}{2}k,\ \frac{5}{2} + i - \frac{1}{2}j,\ \frac{7}{2} + 2i + \frac{1}{2}j\right\};\]
this has image given by
\begin{align*}
\varphi(a)&=\begin{bmatrix}  \frac{3}{2}(\sqrt{2} + 1) &   -\frac{1}{2}(\sqrt{2} + 1)\\
\frac{13}{2}(\sqrt{2} - 1) & \frac{3}{2}(1-\sqrt{2})\end{bmatrix}, &
\varphi(b)&=\begin{bmatrix}   \frac{3}{2}(1-\sqrt{2}) &     \frac{1}{2}(\sqrt{2} - 1)\\
-\frac{13}{2}(\sqrt{2} + 1) &   \frac{3}{2}(\sqrt{2} + 1)\end{bmatrix},\\
\varphi(c)&=\begin{bmatrix} \frac{1}{2}(2\sqrt{2} + 5)  &     -\frac{1}{2}\\
     -\frac{13}{2} & \frac{1}{2}(5-2\sqrt{2})\end{bmatrix}, &
\varphi(d)&=\begin{bmatrix}\frac{1}{2}(4\sqrt{2} + 7) &        \frac{1}{2}\\
      \frac{13}{2} & \frac{1}{2}(7-4\sqrt{2})\end{bmatrix}.
\end{align*}
Conjugating $M$ by $t$ we obtain another maximal order $N$.  Let $U^1(M)$ and $U^1(N)$ denote the groups of norm one quaternions under multiplication in $M$ and $N$ respectively.  Note that their image under $\varphi$ is contained in $\SL_2(\RR)$.

Denote the image of $U^1(M)$ and $U^1(N)$ under $\varphi$ after projecting to $\PSL_2(\RR)$ by $\mathrm{P}M$ and $\mathrm{P}N$ respectively.  Both of these groups are isomorphic to the fundamental group of a genus $2$ surface (this may be verified in Magma).  It is easy to see $\varphi(t)$ commensurates $U^1(M)$ and hence $U^1(M)$ and $U^1(N)$ share a common finite index subgroup.  The intersection $K=\mathrm{P}M\cap \mathrm{P}N$ has index $12$ in both $\mathrm{P}M$ and $\mathrm{P}N$, in particular, $K$ is the fundamental group of a genus $13$ surface.  We compute, using Dehn's algorithm a word in $a,b,c,d$ for $t^{-1}gt$ for each generator $g$ of $K$.  We will denote the subgroup generated by these words $H$ and note that $H^t=K$.

We now build a HNN-extension $\Gamma=\mathrm{P}M\ast_{H^t=K}$.  The group has $5$ generators which (abusing notation) we label $a,\ b,\ c,\ d,\ t$ and admits a presentation with $27$ relations, displayed in \Cref{app:pres}.

\begin{lemma}\label{lem:irrlattice}
$\Gamma$ is an irreducible uniform lattice in $\PSL_2(\RR)\times T_{24}$.
\end{lemma}
\begin{proof}
Since $\Gamma$ is a graph of groups equipped with a morphism to $\PSL_2(\RR)$ such that the vertex stabiliser $\Gamma_v$ is a uniform $\PSL_2(\RR)$ lattice and the stable letter commensurates $\Gamma_v$, it follows $\Gamma$ is a \emph{graph of lattices} in the sense of \Cref{def.gol}.  The two embeddings of the edge group have index $12$ in $\Gamma_v$ so the Bass--Serre tree of $\Gamma$ is $24$-regular.  Thus, $\Gamma$ is a uniform lattice in $\PSL_2(\RR)\times T_{24}$ by \Cref{thm.gol}. The image of the subgroup generated by stable letter $t$ in $\PSL_2(\RR)$ is clearly non-discrete because it is generated by an infinite order elliptic isometry.  The irreducibility now follows from \cite[Proposition~3.4]{Hughes2021a}.
\end{proof}

\begin{lemma}\label{lem:non-resfin}
$\Gamma$ is non-residually finite.
\end{lemma}
\begin{proof}
Because $\Gamma$ is an HNN-extension the first Betti number of $\Gamma$ is at least $1$ (in fact a direct computation yields it is exactly $1$).  Since $\Gamma$ is an irreducible lattice it follows from \cite[Proposition~3.7]{Hughes2021a} that $\Gamma$ is non-residually finite.
\end{proof}

\begin{lemma} \label{lem:trlen-irrational}
The translation lengths of $a$ and $c$ in their action on $\RH^2$ are not rational multiples of each other.
\end{lemma}
\begin{proof}
For a hyperbolic isometry $g$ of $\PSL_2(\RR)$ its translation length is \[ \tau(g)=2\cosh^{-1}\left(\frac{1}{2}\tr(\tilde g)\right) \] where $\tilde g$ is a choice of lift of $g$ to $\SL_2(\RR)$.  It follows that we have \[\tau(a)=2\log\left(\frac{3}{2}+\frac{\sqrt{5}}{2}\right) \quad \text{and} \quad \tau(c)= 2\log\left(\frac{5}{2}+\frac{\sqrt{21}}{2}\right).\]
Suppose that $\frac{p}{q}\tau(a)=\tau(c)$ where $p,q\in\ZZ$, $p,q\geq1$, then we have
\[\left(\frac{3}{2}+\frac{\sqrt{5}}{2}\right)^p=\left(\frac{5}{2}+\frac{\sqrt{21}}{2}\right)^q. \]
The left hand side is always of the form $m_1+m_2\sqrt{5}$ and the right hand side is always of the form $m_3+m_4\sqrt{21}$ for some rational numbers $m_1,\ m_2,\ m_3,\ m_4 > 0$.  This is clearly impossible and we conclude that $\tau(a)$ is not a rational multiple of $\tau(c)$.
\end{proof}

\section{Proof of Theorem A}\label{sec:main}

\begin{thmA}
There exists a non-residually finite torsion-free irreducible uniform lattice $\Gamma<\PSL_2(\RR)\times T_{24}$ such that $\Gamma$ is a hierarchically hyperbolic group but is not biautomatic.
\end{thmA}
\begin{proof}
Let $\Gamma$ be the HNN-extension constructed in \Cref{sec:example}.  Then, $\Gamma$ is an irreducible uniform lattice in $\PSL_2(\RR)\times T_{24}$ by \Cref{lem:irrlattice}, non-residually finite by \Cref{lem:non-resfin}, torsion-free by construction, and a hierarchically hyperbolic group by \Cref{prop.HHG}. It remains to show $\Gamma$ is not biautomatic.

Let $\widehat{G} < \Gamma$ be a vertex stabiliser for the $\Gamma$-action on the Bass--Serre tree $\mathcal{T}_{24}$ of $\Gamma$. By construction, we have $\Gamma = \langle \widehat{G},\widehat{t} \rangle$ for an element $\widehat{t} \in \Gamma$ such that $t := \pi_{\PSL_2(\RR)}(\widehat{t})$ is an infinite order elliptic isometry of $\Htwo$. Moreover, the group $G := \pi_{\PSL_2(\RR)}(\widehat{G})$ is a torsion-free uniform lattice in $\PSL_2(\RR)$, and the projection $\pi_{\PSL_2(\RR)}(\Gamma) = \langle G,t \rangle$ is dense in $\PSL_2(\RR)$ by \Cref{lem:dense}. As $\widehat{G}$ is commensurated in $\Gamma$, it follows that $t^{-1}Gt \cap G$ has finite index in $G$. Let $\Sigma = G \backslash \Htwo$, so that $\Sigma$ is a closed orientable hyperbolic surface and $G \cong \pi_1(\Sigma)$.

Now suppose for contradiction that $(B,\calm)$ is a (uniformly finite-to-one) biautomatic structure on $\Gamma$. By \Cref{prop:quasiconvex}, the subgroup $\widehat{G} < \Gamma$ is $\calm$-quasiconvex; let $(A,\call)$ be the biautomatic structure on $\widehat{G}$ associated to $(B,\calm)$, as given by \Cref{thm:qconvex}. As $\pi_{\PSL_2(\RR)}$ maps $\widehat{G}$ isomorphically to $G$, we will identify $(A,\call)$ with a biautomatic structure on $G$. Consider the function $\normL{-}\colon G \to \RR$, as defined in \Cref{defn:biauto-notation}. By construction, $\normL{-}$ is invariant under conjugacy in $G$, and therefore factors through a function $\ocurves{\Sigma} \to \RR$ which we also denote by $\normL{-}$. Similarly, it follows from \Cref{lem:tau-conjinvariant} that the function $\tau_\call\colon G \to \RR$ (as defined in \Cref{defn:biauto-notation}) factors through a function $\tau_\call\colon \ocurves{\Sigma} \to \RR$.

By \Cref{prop:quasi-smoothing}, the function $\normL{-}\colon \ocurves{\Sigma} \to \RR$ satisfies the join and split quasi-smoothing properties, and therefore, by \Cref{thm:MGT-limit}, the function $\tau_\call'\colon \ocurves{\Sigma} \to \RR$ defined by $\tau_\call'(\gamma) = \lim_{n \to \infty} \normL{\gamma^n}/n$ is homogeneous and satisfies the join and split quasi-smoothing properties. By \Cref{lem:translation-lengths-same}, we have $\tau_\call' = \tau_\call$. Thus, by \Cref{thm:MGT-main}, $\tau_\call\colon \ocurves{\Sigma} \to \RR$ extends to a unique continuous homogeneous function $\tau_\call\colon \ocurrents{\Sigma} \to \RR$, which is also positively linear by \Cref{lem:poslin}.

We now claim that $\tau_\call(\gamma) = \tau_\call(\gamma^{(t)})$ for every $\gamma \in \ocurves{\Sigma}$ (in the notation of \Cref{sec:invariant}). Indeed, let $\gamma \in \ocurves{\Sigma}$, let $g \in G$ be an element corresponding to $\gamma$, let $r \geq 1$ be such that $t^{-1} g^r t \in G$, and let $g_1,\ldots,g_s$ be a right transversal of $t^{-1}Gt \cap G$ in $G$. Then, for each $i$, the measure $\gamma^r(tg_i-)$ is a curve on $\Sigma$ corresponding to the element $(tg_i)^{-1}g^rtg_i \in G$. Furthermore, it follows from \Cref{thm:qconvex} that the restriction of $\tau_\calm\colon \Gamma \to \RR$ to $\widehat{G} \cong G$ coincides with $\tau_\call\colon G \to \RR$; in particular, by \Cref{lem:tau-conjinvariant} we have $\tau_\call((tg_i)^{-1}g^rtg_i) = \tau_\calm((tg_i)^{-1}g^rtg_i) = \tau_\calm(g^r) = \tau_\call(g^r)$. As $\tau_\call\colon \ocurrents{\Sigma} \to \RR$ is positively linear, we then have
\begin{align*}
    r \cdot \tau_\call(\gamma^{(t)}) &= s^{-1} \sum_{i=1}^s r \cdot \tau_\call(\gamma(tg_i-)) = s^{-1} \sum_{i=1}^s \tau_\call(\gamma^r(tg_i-)) = s^{-1} \sum_{i=1}^s \tau_\call((tg_i)^{-1}g^rtg_i) \\
    &= s^{-1} \sum_{i=1}^s \tau_\call(g^r) = \tau_\call(g^r) = \tau_\call(\gamma^r) = r \cdot \tau_\call(\gamma),
\end{align*}
and thus $\tau_\call(\gamma) = \tau_\call(\gamma^{(t)})$, as claimed.

We now apply \Cref{prop.criteria} with $F = \tau_\call$; therefore, there exists a constant $k > 0$ such that $\tau_\call = k \cdot \iota(-,\lambda_\Sigma)$. In particular, since $\iota(\gamma,\lambda_\Sigma) > 0$ for all $\gamma \in \ocurves{\Sigma}$ by \Cref{prop:bonahon-filling}, we have $\frac{\tau_\call(\gamma_1)}{\tau_\call(\gamma_2)} = \frac{\iota(\gamma_1,\lambda_\Sigma)}{\iota(\gamma_2,\lambda_\Sigma)}$ for any $\gamma_1,\gamma_2 \in \ocurves{\Sigma}$. But note that $\iota(\gamma,\lambda_\Sigma)$ is precisely the length of the geodesic representative $\Sone \to \Sigma$ of $\gamma$, which is equal to the translation length of a lift of $\gamma$ in its action on $\Htwo$. In particular, if $\gamma_a,\gamma_c \in \ocurves{\Sigma}$ correspond to the elements $a,c \in G$ appearing in \Cref{lem:trlen-irrational}, we then have $\frac{\tau_\call(\gamma_a)}{\tau_\call(\gamma_c)} \notin \QQ$. This contradicts \Cref{prop:stable-len}, which implies that $\tau_\call$ takes only rational values.
\end{proof}

\AtNextBibliography{\small}
\printbibliography

\pagebreak
\begin{appendix}
\section{A presentation of the group} \label{app:pres}
The group $\Gamma$ constructed in \Cref{sec:example} is generated by $a$, $b$, $c$, $d$, and $t$ subject to the following $27$ relations:
\begin{align*}
a^{-1}  d  c  b  c^{-1}  a  b^{-1}  d^{-1} &= 1,\\
t  d^{-1}  a^2  c  b^{-1}  c^{-1}  t^{-1} &= d,\\
t  d^{-1}  a^2  c  a^{-1} d  t^{-1} &= a^2  c^{-1},\\
t  d^{-1}  a^2  c  b d^{-1} a^{-1} d  t^{-1} &= a  c  b^{-1},\\
t  d^{-1}  a  d  c^{-1}  d^{-1}  a  d  c  a^{-2} d  t^{-1} &= b  a  b  a^{-1},\\
t  d^{-1}  a  d  c^{-1} a^{-1}  d   t^{-1} &= b^2  c^{-1},\\
t  d^{-1}  a  d  b  c^{-1}  a  b^{-1}  c^{-1}  d  b  a^{-1}  c  a^{-1}  c  b^{-1}  c^{-1}  t^{-1} &= b  d  c^{-1}  b^{-1},\\
t  b^2  c^{-1}  a^{-2} d  t^{-1} &= c  b  d^{-1}  a^{-1},\\
t  d^{-2}  a^2 d^{-1} a^{-1} d  t^{-1} &= c^3,\\
t  d^{-1}  a b^{-1} d^{-1} a  d  t^{-1} &= c  d  c^{-1},\\
t  b a^{-1} c  b  a^{-1}  d  b  a^{-1}  c  a^{-1}  c  b^{-1}  c^{-1}  t^{-1} &= a^{-1}  b  c^{-1}  b^{-1},\\
t  d^{-1}  a  d  c  b  a^{-1} d^{-1} a^{-1} d  t^{-1} &= a^{-1}  c  b,\\
t  d^{-1}  a^2  d  c d^{-1} a^{-1} d  t^{-1} &= a^{-1}  d  b^{-1},\\
t  d^{-1}  a  d  c  b^{-1}  c^{-1} a^{-1} d  a^{-2} d  t^{-1} &= b^{-1}  a  b^{-1}  a^{-1},\\
t  d^{-1}  a  d  c  b^{-2}  a^{-2} d  t^{-1} &= b^{-1}  c  b  a^{-1},\\
t  d^{-1}  a  d  c  b^{-1}  c^{-1}  b  a^{-2} d  t^{-1} &= b^{-1}  d  b  a^{-1},\\
t  d^{-1}  a  d^2  b  a^{-1}  c  b^{-2}  c^{-1} d^{-1} a^{-1} d  t^{-1} &= c^{-1}  a^2,\\
t  d^{-1}  a  d  a^{-2}  d  b  a^{-1} d^{-1} a^{-1} d  t^{-1} &= c^{-1}  b^2,\\
t  d^{-1}  a^2  d^{-1}  a  c^{-1}  a  b  c^{-1}  a^{-1}  a^{-1}  d   t^{-1} &= a  b  a  d^{-1}  a^{-1},\\
t  d^{-1}  a^2  d^{-1}  c  b^{-2}  c^{-1} d^{-1} a^{-1} d  t^{-1} &= a  b^2  a,\\
t  d^{-1}  a^2  d^{-1}  a  b^{-3}  c^{-1} d^{-1} a^{-1} d  t^{-1} &= a  b  c  a,\\
t  d^{-1}  a^2  d^{-1}  a  c^{-1}  b^{-1}  c^{-1} d^{-1} a^{-1} d  t^{-1} &= a  b  d  a,\\
t  d^{-1}  a^2  c  b^{-1}  a^{-1}  d  b  a^{-1}  d  b  a^{-1}  c  a^{-1}  c  b^{-1}  c^{-1}  t^{-1} &= a  d  a  c^{-1}  b^{-1},\\
t  d^{-1}  a^2  c  b^{-2}  c^{-1}  d^{-1}  a^2 d^{-1} a^{-1} d  t^{-1} &= a  d  b  c,\\
t  d^{-1}  a^2  c^{-1}  d  b  a^{-1} d^{-1} a^{-1} d  t^{-1} &= a  b^{-1}  a  b,\\
t  d^{-1}  a  d  c^{-1}  d^{-1}  a  c  b^{-1}  a^{-1}  c  b^{-1}  c^{-1}  t^{-1} &= b  c  a  b^{-1},\\
t  d^{-1}  a  d  c^{-1}  d^{-1}  a  d^{-1}  a  d  b a^{-1} d  t^{-1} &= b  c  b  a^{-1}  c^{-1}.
\end{align*}

\end{appendix}

\end{document}
    \hline

%% file: top.tex
\usepackage{amsmath} 
\usepackage{amssymb} 
\usepackage{amsthm} 
\usepackage[english]{babel} 
\usepackage[font=small,justification=centering]{caption} 
\usepackage{subcaption} 
\usepackage[nodayofweek]{datetime}
\usepackage[shortlabels]{enumitem} 
\usepackage[T1]{fontenc} 
\usepackage[utf8]{inputenc} 
\usepackage{ifthen} 
\usepackage{mathabx} 
\usepackage{mathtools} 
\usepackage[dvipsnames]{xcolor} 
\usepackage[pdftex,  colorlinks=true]{hyperref} 
    \hypersetup{urlcolor=RoyalBlue, linkcolor=RoyalBlue,  citecolor=black}
\usepackage{setspace} 
\usepackage{tikz-cd} 
\usepackage{xfrac} 
\usepackage[capitalise]{cleveref}
\usepackage{comment}

\allowdisplaybreaks

\makeatletter
\def\@tocline#1#2#3#4#5#6#7{\relax
  \ifnum #1>\c@tocdepth 
  \else
    \par \addpenalty\@secpenalty\addvspace{#2}%
    \begingroup \hyphenpenalty\@M
    \@ifempty{#4}{%
      \@tempdima\csname r@tocindent\number#1\endcsname\relax
    }{%
      \@tempdima#4\relax
    }%
    \parindent\z@ \leftskip#3\relax \advance\leftskip\@tempdima\relax
    \rightskip\@pnumwidth plus4em \parfillskip-\@pnumwidth
    #5\leavevmode\hskip-\@tempdima
      \ifcase #1
       \or\or \hskip 1em \or \hskip 2em \else \hskip 3em \fi%
      #6\nobreak\relax
    \dotfill\hbox to\@pnumwidth{\@tocpagenum{#7}}\par
    \nobreak
    \endgroup
  \fi}
\makeatother

\newtheorem{thm}{Theorem}[section]
\newtheorem{lemma}[thm]{Lemma}

\newtheorem{prop}[thm]{Proposition}
\newtheorem{conjecture}[thm]{Conjecture}

\newtheorem{question}[thm]{Question}

\Crefmultiformat{thm}{Theorems~#2#1#3}{ and~#2#1#3}{, #2#1#3}{ and~#2#1#3}

\newtheorem{thmx}{Theorem}
\newtheorem{corx}[thmx]{Corollary}

\newtheorem*{thmA}{Theorem A}

\theoremstyle{definition}
\newtheorem{defn}[thm]{Definition}

\DeclareMathOperator{\Aut}{\mathrm{Aut}}

\DeclareMathOperator{\Ker}{\mathrm{Ker}}

\DeclareMathOperator{\tr}{\mathrm{tr}}

\DeclareMathOperator{\Comm}{\mathrm{Comm}}
\DeclareMathOperator{\Isom}{\mathrm{Isom}}

\DeclareMathOperator{\cay}{Cay}

\DeclareMathOperator{\Stab}{Stab}

\newcommand{\cala}{{\mathcal{A}}}

\newcommand{\call}{{\mathcal{L}}}
\newcommand{\calm}{{\mathcal{M}}}

\newcommand{\calt}{{\mathcal{T}}}

\newcommand{\SL}{\mathrm{SL}}
\newcommand{\PSL}{\mathrm{PSL}}

\newcommand{\CAT}{\mathrm{CAT}}

\newcommand{\EE}{\mathbb{E}}
\newcommand{\RH}{\mathbb{R}\mathbf{H}} 
\newcommand{\Htwo}{\RH^2}

\newcommand{\Sone}{\mathbb{S}^1}

\newcommand{\ocurrents}[1]{\mathcal{G}^+(#1)}

\newcommand{\ogeod}[1]{\mathcal{I}^+(#1)}

\newcommand{\odgeod}[1]{\mathcal{D}\ogeod{#1}}

\newcommand{\ocurves}[1]{\mathcal{C}^+(#1)}

\newcommand{\normL}[1]{\left\| {#1} \right\|_{\mathcal{L}}}
\newcommand{\modL}[1]{\left| {#1} \right|_{\mathcal{L}}}

\newcommand{\monoid}[1]{\left\langle\!{\cdot} #1 {\cdot}\!\right\rangle}

\newcommand{\ZZ}{\mathbb{Z}}

\newcommand{\RR}{\mathbb{R}}
\newcommand{\QQ}{\mathbb{Q}}

\newcommand{\dd}{\,\mathrm{d}}

\usepackage{tikz}
\usetikzlibrary{arrows,quotes,decorations.markings}
\tikzstyle{blackNode}=[fill=black, draw=black, shape=circle]

%% file: refs.bib
@misc{Amrhein2021,
      title={Characterizing Biautomatic Groups}, 
      author={Aischa Amrhein},
      year={2021},
      eprint={2105.07509},
      archivePrefix={arXiv},
      primaryClass={math.GR}
}

@incollection {BandeltChepoi2008,
	AUTHOR = {Bandelt, Hans-J\"urgen and Chepoi, Victor},
	TITLE = {Metric graph theory and geometry: a survey},
	BOOKTITLE = {Surveys on discrete and computational geometry},
	SERIES = {Contemp. Math.},
	VOLUME = {453},
	PAGES = {49--86},
	PUBLISHER = {Amer. Math. Soc., Providence, RI},
	YEAR = {2008},
	MRCLASS = {05C12 (05C62 51K05 52-02 54E35)},
	MRREVIEWER = {Mikhail Ostrovskii},
	DOI = {10.1090/conm/453/08795},
	URL = {https://doi.org/10.1090/conm/453/08795},
}

@article {Bass1993,
    AUTHOR = {Bass, Hyman},
     TITLE = {Covering theory for graphs of groups},
   JOURNAL = {J. Pure Appl. Algebra},
  FJOURNAL = {Journal of Pure and Applied Algebra},
    VOLUME = {89},
      YEAR = {1993},
    NUMBER = {1-2},
     PAGES = {3--47},
      ISSN = {0022-4049},
   MRCLASS = {20E08 (57M07)},
  MRNUMBER = {1239551},
MRREVIEWER = {Alexander Lubotzky},
       DOI = {10.1016/0022-4049(93)90085-8},
      URL = {https://doi.org/10.1016/0022-4049(93)90085-8},
}

@article {BenoistOh2007,
    AUTHOR = {Benoist, Yves and Oh, Hee},
     TITLE = {Equidistribution of Rational Matrices in their Conjugacy Classes},
   JOURNAL = {Geom. Funct. Anal.},
    VOLUME = {17},
      YEAR = {2007},
     PAGES = {1--32},
       DOI = {10.1017/s00039-006-0585-4},
}

@article{Bonahon1986,
    author = {Bonahon, Francis},
    journal = {Ann. Math.},
    pages = {71--158},
    title = {Bouts des vari\'et\'es hyperboliques de dimension 3},
    doi = {10.2307/1971388},
    volume = {124},
    year = {1986},
}

@article{Bonahon1988,
    author = {Bonahon, Francis},
    journal = {Invent. Math.},
    pages = {139--162},
    title = {The geometry of Teichm\"uller space via geodesic currents},
    doi = {10.1007/BF01393996},
    volume = {92},
    year = {1988},
}

@article{Bowditch2013,
author={Brian H. Bowditch},
journal={Pacific J. Math.},
title={Coarse median spaces and groups},
volume={261},
issue={1},
pages={53-93},
year={2013},
doi={10.2140/pjm.2013.261.53},
ISSN={0030-8730},
}

@article {AlonsoBridson1995,
    AUTHOR = {Alonso, Juan M. and Bridson, Martin R.},
     TITLE = {Semihyperbolic groups},
   JOURNAL = {Proc. London Math. Soc. (3)},
  FJOURNAL = {Proceedings of the London Mathematical Society. Third Series},
    VOLUME = {70},
      YEAR = {1995},
    NUMBER = {1},
     PAGES = {56--114},
      ISSN = {0024-6115},
   MRCLASS = {20F32 (20F10 53C23 57M07)},
  MRNUMBER = {1300841},
MRREVIEWER = {Athanase Papadopoulos},
       DOI = {10.1112/plms/s3-70.1.56},
       URL = {https://doi.org/10.1112/plms/s3-70.1.56},
}

@misc{ChalopinChepoiGenevoisHiraiOsajda2020,
      title={Helly groups}, 
      author={Chalopin, J\'er\'emie and Chepoi, Victor and Genevois, Anthony and Hirai, Hiroshi and Osajda, Damian},
      year={2020},
      eprint={2002.06895},
      archivePrefix={arXiv},
      primaryClass={math.GR}
}

@book{CoornaertDelzantPapadopoulos1990,
	author = {Coornaert, Michel and Delzant, Thomas and Papadopoulos, Athanase},
	title = {G\'eom\'etrie et Th\'eo\-rie des Groupes},
	publisher = {Springer-Verlag},
	series = {Lecture Notes in Mathematics},
	volume = {1441},
	year = {1990}
}

@article{DescombesLang2016,
	AUTHOR = {Descombes, Dominic and Lang, Urs},
	TITLE = {Flats in spaces with convex geodesic bicombings},
	JOURNAL = {Anal. Geom. Metr. Spaces},
	FJOURNAL = {Analysis and Geometry in Metric Spaces},
	VOLUME = {4},
	YEAR = {2016},
	NUMBER = {1},
	PAGES = {68--84},
	ISSN = {2299-3274},
	MRCLASS = {53C23 (20F65 20F67)},
	MRREVIEWER = {R\'{e}mi Bernard Coulon},
	DOI = {10.1515/agms-2016-0003},
	URL = {https://doi.org/10.1515/agms-2016-0003},
}

@article{Dress1984,
title = {Trees, tight extensions of metric spaces, and the cohomological dimension of certain groups: A note on combinatorial properties of metric spaces},
journal = {Adv. Math.},
volume = {53},
number = {3},
pages = {321--402},
year = {1984},
issn = {0001-8708},
doi = {10.1016/0001-8708(84)90029-X},
url = {https://www.sciencedirect.com/science/article/pii/000187088490029X},
author = {Andreas W. M. Dress}
}

@book {EpsteinEtAl,
    AUTHOR = {Epstein, David B. A. and Cannon, James W. and Holt, Derek F.
              and Levy, Silvio V. F. and Paterson, Michael S. and Thurston,
              William P.},
     TITLE = {Word processing in groups},
 PUBLISHER = {Jones and Bartlett Publishers, Boston, MA},
      YEAR = {1992},
     PAGES = {xii+330},
      ISBN = {0-86720-244-0},
   MRCLASS = {20F10 (03D40 20-02 68Q70)},
  MRNUMBER = {1161694},
MRREVIEWER = {Richard M. Thomas},
}

@incollection {FarbHruskaThomas2011,
    AUTHOR = {Farb, Benson and Hruska, Chris and Thomas, Anne},
     TITLE = {Problems on automorphism groups of nonpositively curved
              polyhedral complexes and their lattices},
 BOOKTITLE = {Geometry, rigidity, and group actions},
    SERIES = {Chicago Lectures in Math.},
     PAGES = {515--560},
 PUBLISHER = {Univ. Chicago Press, Chicago, IL},
      YEAR = {2011},
   MRCLASS = {22F50 (20F65 22E40)},
  MRNUMBER = {2807842},
MRREVIEWER = {Pierre-Emmanuel Caprace},
       DOI = {10.7208/chicago/9780226237909.001.0001},
      URL = {https://doi.org/10.7208/chicago/9780226237909.001.0001},
}

@article{Isbell1964,
author = {Isbell, John R.},
journal = {Comment. Math. Helv.},
pages = {65--76},
title = {Six theorems about injective metric spaces},
doi = {10.1007/BF02566944},
volume = {39},
year = {1964},
}

@article{MunroOsajdaPrzyycki2021,
AUTHOR={Zachary Munro and Damian Osajda and Piotr Przytycki},
TITLE={2-dimensional Coxeter groups are biautomatic},
JOURNAL={Proc. R. Soc. Edinb. A: Math.},
VOLUME = {},
ISSUE = {},
PAGES= {},
year={2021},
FJOURNAL={Proceedings of the Royal Society of Edinburgh Section A: Mathematics},
DOI={10.1017/prm.2021.11},
URL={https://doi.org/10.1017/prm.2021.11},
pubstate={To appear},
}

@book {Gaal1973,
    AUTHOR = {Gaal, Steven A.},
     TITLE = {Linear Analysis and Representation Theory},
    SERIES = {Die Grund\-leh\-ren der mathematischen Wissenschaften},
    VOLUME = {198},
 PUBLISHER = {Springer-Verlag},
      YEAR = {1973},
     PAGES = {x+690},
      ISBN = {978-3-642-80741-1},
       DOI = {10.1007/978-3-642-80741-1},}

@article {GerstenShort1991,
    AUTHOR = {Gersten, Stephen M. and Short, Hamish B.},
     TITLE = {Rational subgroups of biautomatic groups},
   JOURNAL = {Ann. of Math. (2)},
  FJOURNAL = {Annals of Mathematics. Second Series},
    VOLUME = {134},
      YEAR = {1991},
    NUMBER = {1},
     PAGES = {125--158},
      ISSN = {0003-486X},
   MRCLASS = {20M05 (20F10 20F32 57M05)},
  MRNUMBER = {1114609},
MRREVIEWER = {Athanase Papadopoulos},
       DOI = {10.2307/2944334},
      URL = {https://doi.org/10.2307/2944334},
}

@book{GhysHarpe1990,
    editor = {Ghys, Etienne and Harpe, Pierre de la},
    title = {Sur les Groupes Hyperboliques d'apr\`es {Mikhael Gromov}},
    publisher = {Birkh\"auser},
    series = {Progress in Mathematics},
    volume = {83},
    year = {1990}
}

@misc{Haettel2021,
      title={Injective metrics on buildings and symmetric spaces}, 
      author={Thomas Haettel},
      year={2021},
      eprint={2101.09367},
      archivePrefix={arXiv},
      primaryClass={math.GR}
}

@article {HuangOsajda2021,
	AUTHOR = {Huang, Jingyin and Osajda, Damian},
	TITLE = {Helly meets {G}arside and {A}rtin},
	JOURNAL = {Invent. Math.},
	FJOURNAL = {Inventiones Mathematicae},
	VOLUME = {225},
	YEAR = {2021},
	NUMBER = {2},
	PAGES = {395--426},
	ISSN = {0020-9910},
	MRCLASS = {20F65 (20F36 20F67)},
	MRNUMBER = {4285138},
	DOI = {10.1007/s00222-021-01030-8},
}

@phdthesis{HughesThesis,
title={Equivariant cohomology, lattices, and trees},
author={Sam Hughes},
year={2021},
school={School of Mathematical Sciences, University of Southampton}
}

@article{Hughes2021b,
      title={Lattices in a product of trees, hierarchically hyperbolic groups, and virtual torsion-freeness}, 
      author={Sam Hughes},
      year={2022},
      volume={54},
      number={4},
      pages={1413--1419},
      journal={Bull. London Math. Soc.},
      fjournal={Bulletin of the London Mathematical Society},
      doi={10.112/blms.12637}
}

@misc{Hughes2022,
      title={Irreducible lattices fibring over the circle}, 
      author={Sam Hughes},
      year={2022},
      eprint={2201.06525},
      archivePrefix={arXiv},
      primaryClass={math.GR},
}

@misc{Hughes2021a,
      title={Graphs and complexes of lattices}, 
      author={Sam Hughes},
      year={2021},
      eprint={2104.13728},
      archivePrefix={arXiv},
      primaryClass={math.GR}
}

@article {HuangPrytula2020,
    AUTHOR = {Huang, Jingyin and Prytu{\l}a, Tomasz},
     TITLE = {Commensurators of abelian subgroups in {CAT}(0) groups},
   JOURNAL = {Math. Z.},
  FJOURNAL = {Mathematische Zeitschrift},
    VOLUME = {296},
      YEAR = {2020},
    NUMBER = {1-2},
     PAGES = {79--98},
      ISSN = {0025-5874},
   MRCLASS = {20F65 (20F67)},
  MRNUMBER = {4140732},
MRREVIEWER = {Anthony Genevois},
       DOI = {10.1007/s00209-019-02449-9},
      URL = {https://doi.org/10.1007/s00209-019-02449-9},
}

@book {Katok1992,
    AUTHOR = {Katok, Svetlana},
     TITLE = {Fuchsian groups},
    SERIES = {Chicago Lectures in Mathematics},
 PUBLISHER = {University of Chicago Press, Chicago, IL},
      YEAR = {1992},
     PAGES = {x+175},
   MRCLASS = {20H10 (30F35)},
  MRNUMBER = {1177168},
MRREVIEWER = {I. Kra},
      ISBN = {0-226-42583-5},
}

@article{Lang2013,
	AUTHOR = {Lang, Urs},
	TITLE = {Injective hulls of certain discrete metric spaces and groups},
	JOURNAL = {J. Topol. Anal.},
	FJOURNAL = {Journal of Topology and Analysis},
	VOLUME = {5},
	YEAR = {2013},
	NUMBER = {3},
	PAGES = {297--331},
	ISSN = {1793-5253},
	MRCLASS = {20F65 (30L05 52C10)},
	DOI = {10.1142/S1793525313500118},
	URL = {http://dx.doi.org/10.1142/S1793525313500118},
}

@phdthesis{MartinezGranado2020,
	author = {Mart\'inez-Granado, D\'idac},
	title = {Smoothings: A Study of Curve Functionals},
	school = {Indiana University},
	year = {2020}
}

@article{MartinezGranadoThurston2020,
      title={From curves to currents}, 
      author={Mart\'inez-Granado, D\'idac and Thurston, Dylan P.},
      year={2021},
      journal={Forum Math. Sigma},
      fjournal={Forum of Mathematics, Sigma},
      doi={10.1017/fms.2021.68},
      pubstate={To appear}
}

@misc{McCammond2007,
    title={Algorithms, Dehn functions and automatic groups},
    author={John McCammond},
    year={2007},
    url={https://sites.google.com/a/scu.edu/rscott/pggt}
}

@article {NibloReeves1998,
    AUTHOR = {Niblo, Graham A. and Reeves, Lawrence D.},
     TITLE = {The geometry of cube complexes and the complexity of their
              fundamental groups},
   JOURNAL = {Topology},
  FJOURNAL = {Topology. An International Journal of Mathematics},
    VOLUME = {37},
      YEAR = {1998},
    NUMBER = {3},
     PAGES = {621--633},
      ISSN = {0040-9383},
   MRCLASS = {20F32 (20F55 57M50)},
  MRNUMBER = {1604899},
MRREVIEWER = {Darryl McCullough},
       DOI = {10.1016/S0040-9383(97)00018-9},
       URL = {https://doi.org/10.1016/S0040-9383(97)00018-9},
}

@article{OsajdaValiunas2020,
      title={Helly groups, coarsely Helly groups, and relative hyperbolicity}, 
      author={Osajda, Damian and Valiunas, Motiejus},
      year={2020},
      eprint={2012.03246},
      archivePrefix={arXiv},
      primaryClass={math.GR},
      pubstate = {To appear},
      journal={Trans. Amer. Math. Soc.},
      fjournal={Transactions of the American Mathematical Society}
}

@thesis{PetytThesis,
title={On the large-scale geometry of mapping class groups},
author={Petyt, Harry},
year={2022},
school={School of Mathematics, University of Bristol},
}

@article{LearyMinasyan2019,
    AUTHOR = {Leary, Ian J. and Minasyan, Ashot},
     TITLE = {Commensurating {HNN} extensions: nonpositive curvature and
              biautomaticity},
   JOURNAL = {Geom. Topol.},
  FJOURNAL = {Geometry \& Topology},
    VOLUME = {25},
      YEAR = {2021},
    NUMBER = {4},
     PAGES = {1819--1860},
      ISSN = {1465-3060},
   MRCLASS = {20F10 (20E06 20F67)},
  MRNUMBER = {4286364},
       DOI = {10.2140/gt.2021.25.1819},
       URL = {https://doi.org/10.2140/gt.2021.25.1819},
}

@article {CapraceMonod2009b,
    AUTHOR = {Caprace, Pierre-Emmanuel and Monod, Nicolas},
     TITLE = {Isometry groups of non-positively curved spaces: discrete
              subgroups},
   JOURNAL = {J. Topol.},
  FJOURNAL = {Journal of Topology},
    VOLUME = {2},
      YEAR = {2009},
    NUMBER = {4},
     PAGES = {701--746},
      ISSN = {1753-8416},
   MRCLASS = {53C23 (20F67 22F30 53C21 53C24)},
  MRNUMBER = {2574741},
MRREVIEWER = {Anne Thomas},
       DOI = {10.1112/jtopol/jtp027},
      URL = {https://doi.org/10.1112/jtopol/jtp027},}

@article{Valiunas2020,
      title={Isomorphism classification of {L}eary--{M}inasyan groups}, 
      author={Motiejus Valiunas},
      year={2021},
      doi={10.1515/jgth-2021-0042},
      pubstate={To appear},
      journal={J. Group Theory}
}

@misc{Valiunas2021a,
    title={Leary--{M}inasyan subgroups of biautomatic groups},
    author={Motiejus Valiunas},
    year={2021},
    eprint={2104.13688},
    archivePrefix={arXiv},
    primaryClass={math.GR},
    pubstate={To appear},
    journal={Geom. Dedicata},
}

@misc{AbbottBehrstock2019,
      title={Conjugator lengths in hierarchically hyperbolic groups}, 
      author={Carolyn Abbott and Jason Behrstock},
      year={2019},
      eprint={1808.09604},
      archivePrefix={arXiv},
      primaryClass={math.GR}
}

@misc{AbbottNgSpriano2019,
      title={Hierarchically hyperbolic groups and uniform exponential growth}, 
      author={Carolyn Abbott and Thomas Ng and Davide Spriano},
      year={2019},
      eprint={1909.00439},
      archivePrefix={arXiv},
      primaryClass={math.GR}
}

@article {BehrstockHagenSisto2017a,
    AUTHOR = {Behrstock, Jason and Hagen, Mark F. and Sisto, Alessandro},
     TITLE = {Hierarchically hyperbolic spaces, {I}: {C}urve complexes for
              cubical groups},
   JOURNAL = {Geom. Topol.},
  FJOURNAL = {Geometry \& Topology},
    VOLUME = {21},
      YEAR = {2017},
    NUMBER = {3},
     PAGES = {1731--1804},
      ISSN = {1465-3060},
   MRCLASS = {20F36 (20F55 20F65)},
  MRNUMBER = {3650081},
MRREVIEWER = {Nadia Benakli},
       DOI = {10.2140/gt.2017.21.1731},
      URL = {https://doi.org/10.2140/gt.2017.21.1731},
}

@article {BehrstockHagenSisto2017,
    AUTHOR = {Behrstock, Jason and Hagen, Mark F. and Sisto, Alessandro},
     TITLE = {Asymptotic dimension and small-cancellation for hierarchically
              hyperbolic spaces and groups},
   JOURNAL = {Proc. Lond. Math. Soc. (3)},
  FJOURNAL = {Proceedings of the London Mathematical Society. Third Series},
    VOLUME = {114},
      YEAR = {2017},
    NUMBER = {5},
     PAGES = {890--926},
      ISSN = {0024-6115},
   MRCLASS = {20F65 (20F67 57M07)},
  MRNUMBER = {3653249},
MRREVIEWER = {Mahan Mj},
       DOI = {10.1112/plms.12026},
      URL = {https://doi.org/10.1112/plms.12026},
}

@article {BehrstockHagenSisto2019,
    AUTHOR = {Behrstock, Jason and Hagen, Mark and Sisto, Alessandro},
     TITLE = {Hierarchically hyperbolic spaces {II}: {C}ombination theorems
              and the distance formula},
   JOURNAL = {Pacific J. Math.},
  FJOURNAL = {Pacific Journal of Mathematics},
    VOLUME = {299},
      YEAR = {2019},
    NUMBER = {2},
     PAGES = {257--338},
      ISSN = {0030-8730},
   MRCLASS = {20F36 (20F65 20F67)},
  MRNUMBER = {3956144},
MRREVIEWER = {Jiming Ma},
       DOI = {10.2140/pjm.2019.299.257},
      URL = {https://doi.org/10.2140/pjm.2019.299.257},
}

@article {DurhamHagenSisto2017,
    AUTHOR = {Durham, Matthew G. and Hagen, Mark F. and Sisto,
              Alessandro},
     TITLE = {Boundaries and automorphisms of hierarchically hyperbolic
              spaces},
   JOURNAL = {Geom. Topol.},
  FJOURNAL = {Geometry \& Topology},
    VOLUME = {21},
      YEAR = {2017},
    NUMBER = {6},
     PAGES = {3659--3758},
      ISSN = {1465-3060},
   MRCLASS = {20F65 (20F67 30F60)},
  MRNUMBER = {3693574},
MRREVIEWER = {Nadia Benakli},
       DOI = {10.2140/gt.2017.21.3659},
      URL = {https://doi.org/10.2140/gt.2017.21.3659},
}

@article{DurhamMinskySisto2020,
      title={Stable cubulations, bicombings, and barycenters}, 
      author={Matthew G. Durham and Yair N. Minsky and Alessandro Sisto},
      year={2020},
      eprint={2009.13647},
      archivePrefix={arXiv},
      primaryClass={math.GR},
      pubstate={To appear},
      journal={Geom. Topol.}
}

@article{HaettelHodaPetyt2021,
      title={Coarse injectivity, hierarchical hyperbolicity, and semihyperbolicity}, 
      author={Thomas Haettel and Nima Hoda and Harry Petyt},
      year={2020},
      eprint={2009.14053},
      archivePrefix={arXiv},
      primaryClass={math.GR},
      pubstate={To appear},
      journal={Geom. Topol.},
      fjournal={Geometry and Topology},
}

@article {HagenSusse2020,
    AUTHOR = {Hagen, Mark F. and Susse, Tim},
     TITLE = {On hierarchical hyperbolicity of cubical groups},
   JOURNAL = {Israel J. Math.},
  FJOURNAL = {Israel Journal of Mathematics},
    VOLUME = {236},
      YEAR = {2020},
    NUMBER = {1},
     PAGES = {45--89},
      ISSN = {0021-2172},
   MRCLASS = {20F65 (20F36 20F67 57M07)},
  MRNUMBER = {4093881},
       DOI = {10.1007/s11856-020-1967-2},
      URL = {https://doi.org/10.1007/s11856-020-1967-2},
}

@misc{Hamenstaedt2009,
      title={Geometry of the mapping class group II: A biautomatic structure}, 
      author={Ursula Hamenstaedt},
      year={2009},
      eprint={0912.0137},
      archivePrefix={arXiv},
      primaryClass={math.GR}
}

@misc{Spriano2018I,
      title={Hyperbolic HHS I:Factor Systems and Quasi-convex subgroups}, 
      author={Davide Spriano},
      year={2018},
      eprint={1711.10931},
      archivePrefix={arXiv},
      primaryClass={math.GR}
}

@misc{Spriano2018II,
      title={Hyperbolic HHS II: Graphs of hierarchically hyperbolic groups}, 
      author={Davide Spriano},
      year={2018},
      eprint={1801.01850},
      archivePrefix={arXiv},
      primaryClass={math.GR}
}

@misc{RobbioSpriano2020,
      title={Hierarchical hyperbolicity of hyperbolic-2-de\-com\-pos\-able groups}, 
      author={Bruno Robbio and Davide Spriano},
      year={2020},
      eprint={2007.13383},
      archivePrefix={arXiv},
      primaryClass={math.GR}
}

@misc{FedericoRobbio2019,
      title={A refined combination theorem for hierarchically hyperbolic groups}, 
      author={Federico Berlai and Bruno Robbio},
      year={2019},
      eprint={1810.06476},
      archivePrefix={arXiv},
      primaryClass={math.GR}
}

@article{PetytSpriano2020,
      title={Unbounded domains in hierarchically hyperbolic groups}, 
      author={Harry Petyt and Davide Spriano},
      year={2022},
      archivePrefix = {arXiv},
      eprint={2007.12535},
      journal={Groups Geom. Dyn.},
      fjournal={Groups, Geometry, and Dynamics},
      primaryClass={math.GR},
      pubstate={To appear},
}
